\documentclass[11pt,reqno]{amsart}
\usepackage{amsthm, amsmath,amsfonts,amssymb,euscript,hyperref,graphics,color}
\usepackage{graphicx}
\usepackage{comment}
\usepackage{import}
\usepackage{tikz}
\usepackage{latexsym}

\def\ub{\underline{u}}
\def\Lb{\underline{L}}
\def\Cb{\underline{C}} \def\Eb{\underline{E}}
\def\Fb{\underline{F}}
\def\gslash{\mbox{$g \mkern -8mu /$ \!}}

\def\doubleint{\int\!\!\!\!\!\int}
\def\nablaslash{\mbox{$\nabla \mkern -13mu /$ \!}}
\def\laplacianslash{\mbox{$\triangle \mkern -13mu /$ \!}}

\newtheorem*{Main Theorem 1}{Main Theorem 1}
\newtheorem*{Main Theorem 2}{Main Theorem 2}
\newtheorem*{Main Theorem 3}{Main Theorem 3}

\newtheorem*{Main A priori Estimates}{Main A priori Estimates}
\newtheorem*{Existence Theorem for Finite Region}{Existence Theorem for Finite Region}
\newtheorem*{Existence Theorem from Past Null Infinity}{Existence Theorem from Past Null Infinity}
\newtheorem{theorem}{Theorem}[section]
\newtheorem{lemma}[theorem]{Lemma}
\newtheorem{proposition}[theorem]{Proposition}

\newtheorem{remark}[theorem]{Remark}

\numberwithin{equation}{section}

\begin{document}
\title[Large data regime for NLW]{A Large Data Regime for non-linear Wave Equations}

\author[Jinhua Wang]{Jinhua Wang}
\address{Center of Mathematical Sciences, Zhejiang University\\ Hangzhou, China}
\email{youcky0208@gmail.com}

\author[Pin Yu]{Pin Yu}
\address{Mathematical Sciences Center, Tsinghua University\\ Beijing, China}
\email{pin@math.tsinghua.edu.cn}

\thanks{JW is deeply indebted to Professors Dexing Kong and Kefeng Liu for the encouragement and guidance. She would like to thank the Mathematical Sciences Center of Tsinghua
University where the work was partially done during her visit.\\
\indent PY is supported by NSF-China Grant 11101235. He would like to thank Prof. Sergiu Klainerman for the communication of the ideas on the relaxation of the propagation estimates.}

\begin{abstract}
For semi-linear wave equations with null form non-linearities on $\mathbb{R}^{3+1}$, we exhibit an open set of initial data which are allowed to be large in energy spaces, yet we can still obtain global solutions in the future.

We also exhibit a set of localized data for which the corresponding solutions are strongly focused, which in geometric terms means that a wave travels along an specific incoming null geodesic in such a way that almost all of the energy is confined in a tubular neighborhood of the geodesic and almost no energy radiating out of this tubular neighborhood.
\end{abstract}
\maketitle

\section{Introduction}\label{introduction}

In this paper, we study the Cauchy problem for the following semi-linear wave equation on $\mathbb{R}^{3+1}$,
 \begin{equation}\label{Main Equation}
 \Box \varphi =Q(\nabla \varphi, \nabla \varphi),
 \end{equation}
where $Q$ is a null form (see Section \ref{section null form} for definitions) and $\varphi: \mathbb{R}^{3+1} \rightarrow \mathbb{R}$ is a scalar function. The data that we will consider for \eqref{Main Equation} will be some specific large data. In fact, the size of the data is measured by a large parameter $\delta^{-1}$ (where $\delta$ is sufficiently small) in energy spaces. We remark that the results of current work can be easily extended to higher dimensions and to a system of equations with null form nonlinearities in the obvious way.

\subsection{Earlier Results}
We briefly summarize the progress on small data theory for nonlinear
waves related to equations of type \eqref{Main Equation}. Based on
the decay mechanism of the linear waves, we know very well about the
Cauchy problems for \eqref{Main Equation} on Minkowski space-times
$\mathbb{R}^{n+1}$, especially for those with small initial data. In
dimensions four and higher, since the linear wave decay fast enough
(at least at the rate $t^{-\frac{3}{2}}$), the
small-data-global-existence type theorems for \eqref{Main Equation}
hold for generic quadratic nonlinearities (which are not assumed to
be \emph{null}), see the work of Klainerman \cite{K-85}. However, in
$\mathbb{R}^{3+1}$, the linear waves decay slower and the quadratic
nonlinearities control the dynamics of the system. In fact, there
are quadratic forms $Q$ for which a finite time blow-up phenomenon
occurs even for arbitrarily small data. This has been shown by F.
John in \cite{J-79}.

The main breakthrough in understanding the small-data-global-existence result of \eqref{Main Equation} was made by Klainerman \cite{K-84} by introducing the \emph{null condition} on the nonlinearities. Under this condition, Klainerman and Christodoulou \cite{Ch-86} independently proved that small initial data lead to global in time classical solutions. Their proofs are different in nature. Klainerman's approach makes use of the full conformal symmetries of $\mathbb{R}^{3+1}$ through the vector fields, while Christodoulou's idea is to use the conformal compactification of $\mathbb{R}^{3+1}$. Nevertheless, their proofs rely essentially on the special cancelations of the null form nonlinearities, which are absent for generic quadratic nonlinearities.

The cancelations of null forms has far-reaching implications for other types of hyperbolic equations. Although many hyperbolic equations do not in general have a null quadratic form type nonlinearity, the estimates for the nonlinear terms follow more or less the similar philosophy: if one term behaves badly (i.e., large in some suitable norms in most cases) in the nonlinearities, it must be coupled with the a good (i.e., with a much better or smaller estimates) term. Thus, we hope that the good terms are strong enough to absorb the large contributions from the bad terms. One major application of this idea in general relativity appears in Christodoulou-Klainerman's proof of non-linear stability of Minkowski space-time \cite{Ch-K}. They observed that a bad (decay worse) component of Weyl curvature is always coupled to either a good connection coefficient or a good curvature component, thus in most cases the bad components do not really affect the long time behavior of the gravitational waves.\\

Although all of the aforementioned results require the initial data be sufficiently small, the idea of using cancelations from null forms still can be used to handle certain large data problems. We shall briefly describe two very recent works on the dynamics of vacuum Einstein field equations in general relativity.

In his seminal work \cite{Ch-08}, Christodoulou discovered a remarkable mechanism responsible for the dynamical formation of black holes.  For some carefully chosen initial data (which
give an open set of the Sobolev space on a outgoing null hypersurface), called \emph{short pulse} data in \cite{Ch-08}, he proved that a black hole (more precisely, a trapped surface)
can form along the evolution due to the focusing of the gravitational waves. Besides its significance in physics, this result is truly remarkable from a PDE perspective, because the
result is for large data (roughly speaking, small data for Einstein equations in general relativity would lead to a space-time close to Minkowski space-time. So for small data, we do not
expect black holes). One of the key observations used repeatedly in \cite{Ch-08} is still related to the philosophy of null forms: we do have many bad (large) components in the estimates,
but all of them must come with good (small) components to make the estimates work.

In \cite{K-R-09}, Klainerman and Rodnianski extended and significantly simplified Christodoulou's work. A key ingredient in their paper is the relaxed propagation estimates, namely, if one enlarges the admissible set of initial conditions, the corresponding propagation estimates are much easier to derive. They reduced the number of derivatives needed in the estimates from two derivatives on curvature (in Christodoulou's proof) to just one. We should note that the direct consequence of the simpler proof of Klainerman-Rodnianski yields results weaker than those obtained by Christodoulou. In fact, within those more general initial data set, they can only show long time existence results for vacuum Einstein field equations; nevertheless, once such existence results are obtained, one can improve them by assuming more on the data, say, consistent with Christodoulou's assumptions and then one can derive Christodoulou's results in a straightforward manner.\\

The results of this paper are strongly motivated by \cite{Ch-08} and the proofs are very much inspired by \cite{K-R-09}. In particular, the choice of initial data will be analogous to the short pulse ansatz in \cite{Ch-08}; the proof will rely on a relaxed version of energy estimates similar to the relaxation of the propagation estimates in \cite{K-R-09}. We also have to mention another work \cite{K-R-10} of Klainerman and Rodnianski where they managed to localize the data for Einstein equations to show the dynamical formation of locally trapped surfaces. One of our main results here concerning the strongly focused waves is motivated by this work. Roughly speaking, it asserts that if the wave initially concentrates around a given point in a specific way, then it will be confined in a tubular neighborhood of an incoming null geodesic and there is only a negligible amount of energy dispersing out of this neighborhood. It is precisely in this sense that we say the wave is strongly focused. It seems to the authors that this result is new even for linear wave equations.

\subsection{Main Results}
We first state the main theorems of the paper:
\begin{Main Theorem 1}
For any given positive number $E_0 >0$, there exists a smooth initial data set $(\varphi^{(0)}, \varphi^{(1)})$ for
 \begin{equation}\label{Main Equation with Cauchy data}
 \begin{split}
 \Box \varphi &=Q(\nabla \varphi, \nabla \varphi),\\
 (\varphi, \partial_t \varphi)|_{t=0} &= (\varphi^{(0)}, \varphi^{(1)}),
 \end{split}
 \end{equation}
where $Q$ is a null form, so that the energy
\begin{equation*}
\text{Energy}_{(1)}(\varphi^{(0)}, \varphi^{(1)}) = \frac{1}{2} \int_{\mathbb{R}^3} |\nabla\varphi^{(0)}|^2 + |\varphi^{(1)}|^2 d x_1 d x_2 d x_3 \geq E_0,
\end{equation*}
and this data set leads to a classical smooth solution of the above equation. Moreover, the future life-span of the solution is $[0,\infty)$.
\end{Main Theorem 1}
\begin{remark}
Moreover, if for $k\in \mathbb{N}$ we define the $k$-th order energy of the data as
\begin{equation*}
\text{Energy}_{(k)}(\varphi^{(0)}, \varphi^{(1)}) = \frac{1}{2} \int_{\mathbb{R}^2} |\nabla^{k}\varphi^{(0)}|^2 + |\nabla^{k-1}\varphi^{(1)}|^2 d x_1 d x_2 d x_3,
\end{equation*}
we can show that
\begin{equation*}
\text{Energy}_{(k)}(\varphi^{(0)}, \varphi^{(1)}) \geq \delta^{-(k-1)},
\end{equation*}
where $\delta$ is a small positive parameter. We note in passing that the higher order energies can be extremely large.
\end{remark}

In the course of proving the above theorem, we will derive two other results which are of independent interest. To facilitate the statement of these two results, we introduce a bit of notation.\\

We review some geometric constructions on Minkowski space $\mathbb{R}^{3+1}$. Besides the standard coordinates $(t, x_1,x_2,x_3)$, we shall mainly use the null-polar coordinates $(u,\ub,\theta)$. We recall the definition for null-polar coordinates. Let
\begin{equation*}
r = \sqrt{x_1^2 + x_2^2 +x_3^2}
\end{equation*}
be the spatial radius function. Two optical function $u$ and $\ub$ are defined by
\begin{equation*}
u = \frac{1}{2}(t-r) \quad \text{and} \quad \ub = \frac{1}{2}(t+r).
\end{equation*}
The angular argument $\theta$ denotes a point on the unit sphere $\mathbb{S}^2 \subset \mathbb{R}^3$.

The past null infinity $\mathcal{I}\,^-$ of $\mathbb{R}^{3+1}$ can be represented by the collection of past-pointing outgoing null lines. Therefore, $\mathcal{I}\,^-$ is parameterized by $(\ub, \theta) \in \mathbb{R} \times \mathbb{S}^2$. We also use $C_{c}$ to denote the level surface $u=c$, where c is a constant; similarly, $\Cb_{\ub}$ denotes a level set of $\ub$. Their intersection $C_u\cap \Cb_{\ub}$ will be a two sphere denoted by $S_{\ub,u}$.

We illustrate these definitions in the following pictures:

\includegraphics[width=5in]{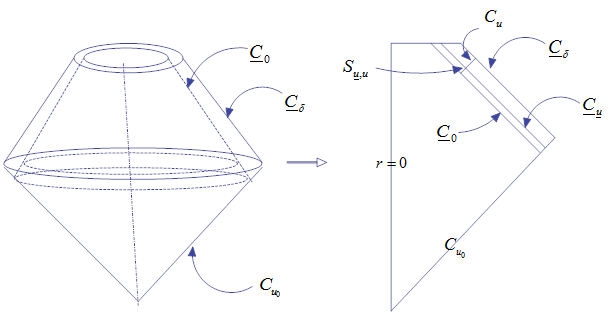}

We remark that, in Section \ref{Section Initial Data in Short Pulse Regime} and Section \ref{Section An A priori Estimates up to Third Derivatives}, which are the technical heart of the paper, the parameter $u$ will be confined in the interval $[u_0, -1]$, where $u_0$ is a large negative number. The parameter $\ub$ is confined in $[u_0, \delta]$, where $\delta$ is small positive parameter, which will be determined later. To simplify our presentation, we will ignore the $\theta$ directions in our pictures. Thus, instead of the left picture above, we will adopt the right picture; as such, the sphere $S_{\ub, u}$ is represented by a single point in the picture.\\

We use $L$ and $\Lb$ to denote the following future-pointed null vector fields:
\begin{equation*}
 L = \partial_t + \partial_r  \quad  \text{and} \quad \Lb = \partial_t - \partial_r.
\end{equation*}
We shall use $\nablaslash$ to denote the intrinsic covariant derivative on $S_{\ub,u}$. It is the restriction of the usual covariant derivative of $(\mathbb{R}^{3+1},g)$ to $S_{\ub,u}$, where $g$ is the standard flat Lorentzian metric on $\mathbb{R}^{3+1}$.

As usual, we use $\mathfrak{so}(3)$ to denote the Lie algebra of the rotation group $SO(3)$ which acts on $\mathbb{R}^{3+1}$ in a canonical way. We choose a set of generators $\{\Omega_1, \Omega_2, \Omega_3\}$ of $\mathfrak{so}(3)$ in the usual way, namely,
\begin{equation*}
\begin{split}
 \Omega_1 &= x_2 \frac{\partial}{\partial x_3} - x_3 \frac{\partial}{\partial x_2}, \\
\Omega_2 &= x_3 \frac{\partial}{\partial x_1} - x_1 \frac{\partial}{\partial x_3}, \\
\Omega_3 &= x_1 \frac{\partial}{\partial x_2} - x_2 \frac{\partial}{\partial x_1}.
\end{split}
\end{equation*}
In what follows, we shall use $\Omega$ to denote a generic $\Omega_i$ and use $\Omega^2$ to denote a generic operator of the form $\Omega_i \Omega_j$, and so on. For a given function $\phi$, we use $|\Omega \phi|$ to denote the sum $\sum_{i=1}^3 |\Omega_i \phi|$ and use $|\Omega^2 \phi|$ to denote the sum $\sum_{1\leq i,j \leq 3} |\Omega_i \Omega_j \phi|$, and so on.

We observe that there are two positive constants $C_1$ and $C_2$, such that on any $S_{\ub,u}$ we have
\begin{equation*}
C_1 |u| |\nablaslash \phi|\leq |\Omega \phi| \leq C_2 |u| |\nablaslash \phi|.
\end{equation*}
For the sake of simplicity, we write this inequality as $|\Omega \phi| \sim |u| |\nablaslash \phi|$. The proof is straightforward: we first check it on the unit sphere and then we use the scaling to get the factor $|u|$. In general, for a given $k \in \mathbb{Z}_{\geq 0}$, we have
\begin{equation}\label{Compare Omega and nablaslash}
 |\Omega^k \phi| \sim |u|^k |\nablaslash^k \phi|.
\end{equation}

The geometric picture all the way to null infinities is usually represented by the Penrose diagram of $(\mathbb{R}^{3+1},g)$:

\includegraphics[width=4in]{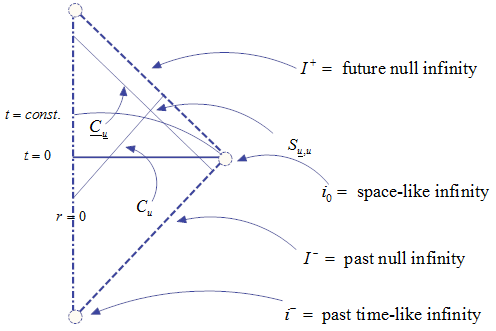}

For many situations in the current work, one has to study the Goursat problem of \eqref{Main Equation}, namely the characteristic problem, instead of the Cauchy problem of \eqref{Main Equation}. Therefore, the null hypersurfaces are the main geometric objets in sequel.

Our second main theorem is a semi-global existence result for a Goursat problem of \eqref{Main Equation} where the data is described on a part of a virtual null hypersurface, i.e., the past null infinity $\mathcal{I}\,^-$. More precisely, the initial data of \eqref{Main Equation} will be a radiation field given on the subset
\begin{equation*}
\mathcal{I}\,^-_\delta = \{(\ub, \theta) \in \mathcal{I}\,^- \,\,|\,\, \ub \leq \delta\},
\end{equation*}
in the asymptotic sense, where $\delta$ is a positive small parameter determined later. Explicitly, the datum is given by a smooth function
\begin{equation*}
  \varphi_{-\infty}: \mathcal{I}\,^-_\delta \rightarrow \mathbb{R}, \qquad (\ub, \theta) \mapsto \varphi_{-\infty}(\ub, \theta),
\end{equation*}
and we require the solution of \eqref{Main Equation} to obey the asymptotic condition,
\begin{equation*}
 \varphi(u,\ub, \theta) \sim \frac{1}{|u|} \varphi_{-\infty}(\ub, \theta) + o(\frac{1}{|u|})
\end{equation*}
We remark that, for linear waves, $\dfrac{1}{|u|}$ is the expected decay rate towards past null infinity. We also remark that, when we speak about the smallness or largeness of the data, we always mean the smallness or largeness of the radiation filed $\varphi_{-\infty}$ instead of $\varphi$ itself (which vanishes on $\mathcal{I}\,^-$).

In this work, we require the initial data $\varphi_{-\infty}$ to have the following form,
\begin{equation}\label{data at infinity}
  \varphi_{-\infty}(\ub,  \theta) =
\begin{cases}
0, &\text{if } \ub \leq 0, \\
\delta^{\frac{1}{2}}\psi_0 (\dfrac{\ub}{\delta}, \theta), &\text{if } 0\leq \ub \leq \delta,
\end{cases}
\end{equation}
where $\psi_0:(0,1) \times \mathbb{S}^2 \rightarrow \mathbb{R}$ is a fixed compactly supported smooth function. More generally, we can put $\varphi_{-\infty}$ in an open set of certain Sobolev spaces defined on $\mathcal{I}\,^-$. We do not pursue this point at the moment and we shall revisit it at the last section of the paper.

The datum given in the above form is called a \emph{short pulse}, a name coined by Christodoulou in \cite{Ch-08}. In his work, he prescribes the shear (more precisely, the conformal geometry) of the initial null hypersurface in a similar form. The shear in the situation of \cite{Ch-08} is exactly the initial data for Einstein vacuum equation.

One may argue that the datum \eqref{data at infinity} is small when $\delta$ is small, at least pointwisely. In fact, the $L^\infty$ norm is irrelevant to the equation \eqref{Main Equation} since we may always add a constant to get a new solution. The size of the datum should be measured at the level of derivatives. While the $\partial_{\ub}$ derivative of the datum can be extremely \emph{large} if $\delta$ is small. In sequel, we shall see that the energy of $\varphi$ will be comparable to $1$ and the higher order energy of $\varphi$ will be comparable to some $\delta^{-a}$ with $a>0$. Therefore, the datum is no longer small in energy spaces.

Before the statement of the second main theorem, we recall that the domain of dependence $\mathcal{D}^+(\mathcal{I}\,^-_\delta)$ of $\mathcal{I}\,^-_\delta$ is the backwards light-cone in $\mathbb{R}^{3+1}$ with vertex at $(\delta, 0,0,0)$.
\begin{Main Theorem 2}\label{Main Theorem 2}
For the non-linear wave equation
 \begin{equation*}
 \Box \varphi =Q(\nabla \varphi, \nabla \varphi),
 \end{equation*}
where $Q$ is a null form, with the following asymptotic characteristic initial datum on the out-going null infinity $\mathcal{I}\,^-$:
\begin{equation}\label{Initial-data-from-past-null}
 \lim_{u\rightarrow \infty} |u|\varphi(u,\ub,\theta) = \varphi_{-\infty}(\ub,  \theta)  \quad \text{ for all } (\ub,  \theta),
\end{equation}
where $\varphi_{-\infty}\in C^\infty(\mathcal{I}\,^-_\delta)$ is given by
\begin{equation*}
  \varphi_{-\infty}(\ub,  \theta) =
\begin{cases}
0, &\text{if } \ub \leq 0, \\
\delta^{\frac{1}{2}}\psi_0 (\dfrac{\ub}{\delta}, \theta), &\text{if } 0\leq \ub \leq \delta.
\end{cases}
\end{equation*}
Here $\psi_0:(0,1) \times \mathbb{S}^2 \rightarrow \mathbb{R}$ is a fixed compactly supported smooth function. If $\delta$ is sufficiently small, there exists one and only one classical solution $\varphi$ on $\mathcal{D}^+(\mathcal{I}\,^-_\delta) \cap \{t \leq -1\}$, such that the radiation field of $\varphi$ is exactly $\varphi_\infty$, i.e. as described in \eqref{Initial-data-from-past-null}.
\end{Main Theorem 2}
The theorem can also be depicted as follows (notice that the region $\mathcal{D}^+(\mathcal{I}\,^-_\delta) \cap \{t \leq -1\}$ is enclosed by four red lines):

\includegraphics[width=4.5in]{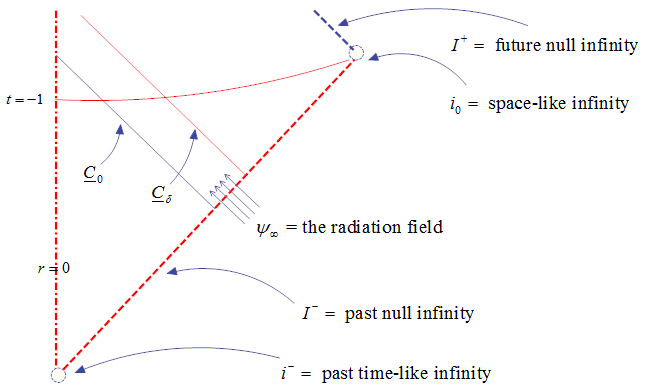}

In the course of the proof of Main Theorem 2, we shall see that the energy flux through $\Cb_\delta$ (if fact, the flux will be bounded above by $\delta^a$ for some $a>0$). Thus, almost all of the energy is confined in $\mathcal{D}^+(\mathcal{I}\,^-_\delta) \cap \{t \leq -1\}$ and there is very little energy radiates to the future null infinity. Intuitively, the waves travel from past null infinity in the incoming direction all the way up to the finite time $t =-1$ with almost no loss of energy. This manifests a strong focusing effect of the nonlinear wave for a special choice of data.\\

We now turn to the last main theorem of the paper where the data are prescribed on a fixed finite null hypersurface $C_{u_0}$ instead of on past null infinity. Only for this third main theorem, we can fix a finite $u_0 \leq -2$, say $u_0 = -10$.

The initial out-going null hypersurface is
\begin{equation*}
C_{u_0} = \{(u,\ub, \theta)| u_0 \leq \ub \leq \delta\},
\end{equation*}
and we also require the initial data $\varphi_{u_0}$ to behave like a pulse:
\begin{equation}
  \varphi_{u_0}(\ub,  \theta) =
\begin{cases}
0, &\text{if } u_0 \leq \ub \leq 0, \\
\dfrac{\delta^{\frac{1}{2}}}{|u_0|}\psi_0 (\dfrac{\ub}{\delta}, \theta), &\text{if } 0\leq \ub \leq \delta.
\end{cases}
\end{equation}
Beside the above requirement, we impose the condition that $\psi_0$ is localized for the angular argument $\theta$, i.e., there is a fixed angle $\theta_0 \in \mathbb{S}^2$, so that the support of $\psi_0$ is contained in a geodesic ball $B_{\delta^{\frac{1}{2}}}(\theta_0)$ centered at $\theta_0$ of radius $\delta^{\frac{1}{2}}$ on $\mathbb{S}^2$. Moreover, we can require $\psi_0$ satisfy the following estimates on $\mathbb{S}^2$,
\begin{equation*}
 \sum_{k=1}^4 \delta^{\frac{k-1}{2}}\|\nablaslash^k \psi_0\|_{L^\infty(\mathbb{S}^2)} \lesssim 1,
\end{equation*}
where $\nablaslash$ is the covariant derivative for the standard metric on $\mathbb{S}^2$. We call such initial datum $\varphi_{u_0}(\ub,  \theta)$ a \emph{short pulse localized in $B_{\delta^{\frac{1}{2}}}(\theta_0)$}.
\begin{Main Theorem 3}\label{Main Theorem 3}
Consider the non-linear wave equation
\begin{equation*}
\Box \varphi =Q(\nabla \varphi, \nabla \varphi),
\end{equation*}
with characteristic initial data given by a short pulse data localized in $B_{\delta^{\frac{1}{2}}}(\theta_0).$
If $\delta$ is small enough, there exists one and only one classical solution $\varphi$ on $\mathcal{D}^+(C_{u_0}) \cap \{u \leq -1\}$.

Moreover, if one uses $C^o_{u}\triangleq u \times [0, \delta] \times B_{\delta^{\frac{1}{2}}}(\theta_0)$ to denote a small open set on the outgoing null hypersurface $C_u$,  the energy is almost localized in $\cup_{u \in [u_0, -1]} C^o_{u}$ which is a tubular neighborhood of some incoming null geodesic parameterized by $u \in [u_0,-1]$ with fixed $\ub$ and $\theta=\theta_0$. More precisely, for all $u \in [u_0,-1]$, the incoming energy outside the small neighborhood $C^o_{u},$ is bounded as follows
\begin{equation*}
 \int_{C_u -  C^o_{u}}|L \varphi|^2 + |\nablaslash \varphi|^2 \lesssim \delta^2.
\end{equation*}
And the energy inside $C^o_{u}$ is almost conserved,
\begin{equation*}
 |\int_{C^o_{u}} |L \varphi|^2 + |\nablaslash \varphi|^2 -  \int_{C^o_{u_0}} |L \varphi|^2 + |\nablaslash \varphi|^2 |\lesssim \delta.
\end{equation*}
\end{Main Theorem 3}
The above theorem can also be depicted as follows

\includegraphics[width=5in]{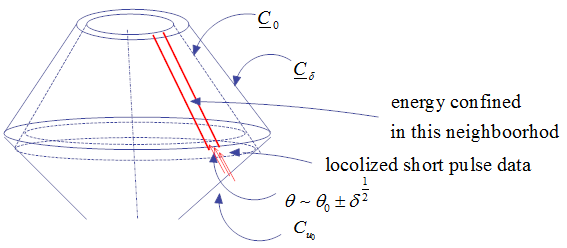}

We can also show that there is almost no energy radiating out through the incoming null hypersurface $\Cb_{\delta}$. Quantitatively, we have
\begin{equation*}
 \int_{\Cb_\delta} |\Lb \varphi|^2 + |\nablaslash \varphi|^2 \lesssim \delta.
\end{equation*}
The above estimates are also true for higher order flux and this will be clear later.

In the proof, we shall see that we can prove a stronger version which we do have $L^\infty$ control on all the first derivatives of $\varphi$. On the final outgoing null hypersurface $C_{-1}$, the energy is mostly contributed by $L\varphi$ and the other components are small (measured in terms of a positive power of $\delta$), in fact, we can show that $L \varphi$ is almost unchanged along the evolution, i.e., $|L\varphi(-1, \ub, \theta)| \sim |L\varphi(u_0, \ub, \theta)|.$  More precisely,
\begin{equation*}
 |L\varphi(-1, \ub, \theta)-L\varphi(u_0, \ub, \theta)|\lesssim \delta^{\frac{1}{2}}.
\end{equation*}
Since we take a short pulse data localized in $B_{\delta^{\frac{1}{2}}}(\theta_0)$, roughly speaking, we have
\begin{equation*}
L \varphi (u_0, u, \theta) =
\begin{cases}
 \delta^{-\frac{1}{2}}, \quad & \text{ for } \theta \in B_{\delta^{\frac{1}{2}}}(\theta_0),\\
0, \quad & \text{ for } \theta \notin B_{\delta^{\frac{1}{2}}}(\theta_0).
\end{cases}
\end{equation*}
Thus, for the final surface $C_{-1}$, roughly speaking, we actually have
\begin{equation*}
L \varphi (-1, u, \theta) \sim
\begin{cases}
 \delta^{-\frac{1}{2}}, \quad & \text{ for } \theta \in B_{\delta^{\frac{1}{2}}}(\theta_0),\\
\delta^{\frac{1}{2}}, \quad & \text{ for } \theta \notin B_{\delta^{\frac{1}{2}}}(\theta_0).
\end{cases}
\end{equation*}
We then integrate those $L^\infty$ estimates to derive the desired control on energy.

Therefore, we have a concentration phenomenon for the a special class of solutions of \eqref{Main Equation} and we say that the solution constructed out from the short pulse data localized in some small spherical sector are \textit{strongly focused}.  As we remarked before, the third theorem appears to be new even for linear wave equations.

\subsection{Comments on the Proof}
We would like to address now the motivations for and difficulties in various estimates leading to the theorem, and then we give an outline of the proof.\\

We first explain the idea of the relaxation for the energy estimates.  This is done in Section \ref{Section An A priori Estimates up to Third Derivatives} which provides an a priori estimate up to four derivatives for \eqref{Main Equation} with short pulse initial data. It is the key ingredient for the whole paper.

The proof is based on the usual vector field method to obtain these estimates. We compute the amplitudes of $L\varphi$ and $\nablaslash \varphi$ on $C_{u_0}$ where the data are given. Roughly speaking, the quantitative estimates look like
\begin{equation}\label{eq 1}
\begin{split}
 \|L\varphi\|_{L^\infty(C_{u_0})} &\sim \delta^{-\frac{1}{2}} ,\\
 \|\nablaslash \varphi\|_{L^\infty (C_{u_0})} &\sim \delta^{\frac{1}{2}}.
\end{split}
\end{equation}
When one derives energy estimates for $\varphi$, a natural choice of the multiplier vector field would be $L$, thus, the energy flux would be
\begin{equation}\label{eq 2}
 \int_{C_u} |L \varphi|^2 + \int_{C_{\ub}} |\nablaslash \varphi|^2.
\end{equation}
If we stick to \eqref{eq 1} as an ideal propagation estimates without any relaxation, we expect
\begin{equation*}
\begin{split}
 \int_{C_u} |L \varphi|^2 &\sim 1,\\
\int_{\Cb_{\ub}} |\nablaslash \varphi|^2 & \sim \delta.
\end{split}
\end{equation*}
Therefore, the flux term \eqref{eq 2} will only yield the bound for $ L\varphi$ and not for $\nablaslash \varphi$, because the bound for $\nablaslash \varphi$ is too small compared to that of $L \varphi$ estimated in this way. In other words, in the end, we do not expect to close the argument if we perform a standard bootstrap argument. Of course, this is due to the choice of the initial data: the short pulse data do not respect the natural scaling of the wave equation!

To resolve this difficulty, one has to relax the estimates for $\nablaslash \varphi$, namely, although the size of the initial data suggest the amplitude $\nablaslash \varphi$ behaves like $\delta^{\frac{1}{2}}$, we pretend it amplitude is worse to match \eqref{eq 1}. Therefore, in the energy estimates, we only expect $\nablaslash \varphi$ to behaves like
\begin{equation}\label{eq 3}
 \|\nablaslash \varphi\|_{L^\infty (C_{u})} \sim 1.
\end{equation}
It turns out that, under this weak assumption, we can still close a bootstrap argument to derive energy estimates. Moreover, once the bootstrap argument is closed, if we can afford one more derivative, we can retrieve the stronger estimates \eqref{eq 1} for $\nablaslash \varphi$. This will be proved later.\\

The second difficulty is about the number of derivatives needed for the a priori estimates. Instead of four derivatives, we may attempt to use three derivatives in Section \ref{Section An A priori Estimates up to Third Derivatives}, since this is still good to control $L^\infty$ norm of one derivatives via Sobolev inequalities. This does not work in an obvious way and the reason is as follows: when one derives estimates for third derivatives, we can use the information already obtained for the first and second derivatives, but $L \varphi$ (the worst term) still consists a third derivative term. Thus, we can not reduce the nonlinear term to a linear one. But this is completely different if we use four derivatives: when one tries to control fourth derivatives of the solution, we must have already obtained estimates up to three derivatives. Thus, the control of $L\varphi$ is then independent of the fourth derivatives. Hence, this reduces the nonlinear term to the linear case where the Gronwall's inequality can used to absorb all the bad terms.

We also point out that the second difficulty is also related to the relaxation of the energy estimates. If we use only three derivatives, for some null forms, say $Q_{0j}$, it leads to a nonlinear term of the form $\nablaslash^3 \varphi \cdot L \varphi$. As we commented in last paragraph, we do not have linear control on the $L^\infty$ norm of $L\varphi$, and yet because we use a relaxed estimates, the control of $\nablaslash^3 \varphi$ is not good enough to compensate for the large amplitude of $L \varphi$. This will lead to a large nonlinear term which can not be controlled.\\

We now outline the proof of our Main Theorem 1. The parameter $u_0$ is a large negative number which will be eventually sent to $-\infty$. The following picture helps to understand the structure of the proof:

\includegraphics[width=5.5 in]{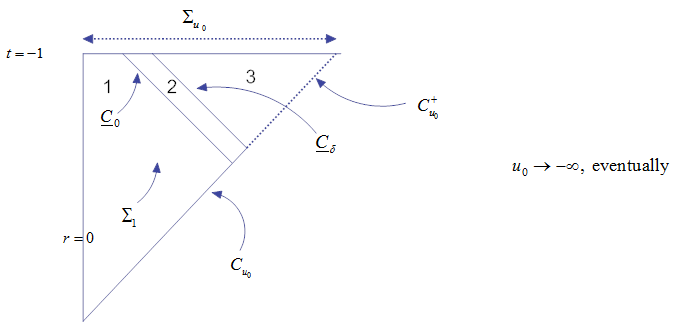}

\begin{itemize}
\item Step 1. We prescribe initial datum on the null hypersurface $C_{u_0}$ where $u_0 \leq \ub \leq \delta$. When $u_0 \leq \ub \leq 0$, the datum is trivial, therefore the solution in Region $1$ in the picture is a constant map. When $0 \leq \ub \leq \delta$, the datum will be prescribed in a specific form (see Section \ref{Section Initial Data in Short Pulse Regime} for detailed account):
    \begin{equation*}
 \varphi(\ub, u_0, \theta) = \delta^{\frac{1}{2}}\psi_0 (\frac{\ub}{\delta}, \theta),
\end{equation*}
with the energy approximately equal to $E_0$. We then show that we can construct a solution in Region $2$ in the picture.

\item Step 2. From the first step, we actually can show that the restriction of the solution already constructed to $\Cb_{\delta}$ are small in energy norms. On $C_{u_0}^+$, which is defined by
    \begin{equation*}
    C_{u_0}^+ = \{ p \in C_{u_0}^+ | \delta \leq \ub(p), t(p) \leq -1 \},
    \end{equation*}
we extend the datum (from Step 1) by zero. Therefore, the datum is also small on $C_{u_0}^+$. We can now use $\Cb_{\delta}$ and $C_{u_0}^+$ as initial hypersurfaces to solve a small data problem to construct solution in Region $3$ in the picture.

\item  Step 3.  We patch the solutions in Regions $1,2$ and $3$ to get one single solution in above picture and then restrict it on the surface $\Sigma_{u_0}$. We then let $u_0$ go to $-\infty$ and use the Arzela-Ascoli lemma to get a solution all the way up to past null infinity. The restrictions to $\Sigma_{u_0}$ then yields a subsequence which converges to a Cauchy data. Finally, we can reverse and shift the time to complete the proof of Main Theorem 1.
\end{itemize}

We remark that Step 1 is the most difficult part since the datum is no longer small and we have to carefully deal with the cancelations from null forms and the profile of the data. Steps 2 and 3 are more or less standard.

\section{Preliminaries}\label{Section Preliminaries}

\subsection{Energy Estimates Scheme}\label{Energy estimates scheme}
Let $\phi$ be a solution for the following non-homogenous wave equation on $\mathbb{R}^{3+1}$, i.e.,
\begin{equation}\label{non-homogenous wave equation}
 \Box \phi = \Phi.
\end{equation}
We define the energy momentum tensor associated to $\phi$ to be
\begin{equation*}
 \mathbb{T}_{\alpha \beta}[\phi] = \nabla_\alpha \phi \nabla_\beta \phi -\frac{1}{2}g_{\alpha \beta} \nabla^{\mu} \phi \nabla_{\mu} \phi.
\end{equation*}
This tensor is symmetric and it enjoys the following divergence identity,
\begin{equation}\label{divergence of T}
 \nabla^\alpha  \mathbb{T}_{\alpha \beta}[\phi] = \Phi \cdot \nabla_\beta \phi.
\end{equation}

Given a vector field $X$, which is usually called a \emph{multiplier vector field}, the associated energy currents are defined as follows
\begin{equation*}
\begin{split}
 J^{X}_{\alpha}[\phi] &= \mathbb{T_{\alpha\mu}}[\phi]X^{\mu},\\
 K^X [\phi] &= \mathbb{T}^{\mu\nu}[\phi]\, ^{(X)}\pi_{\mu \nu},
\end{split}
\end{equation*}
where the deformation tensor $^{(X)}\pi_{\mu \nu}$ is defined by
\begin{equation}
  ^{(X)}\pi_{\mu \nu} = \frac{1}{2} \mathcal{L}_X g_{\mu \nu} = \frac{1}{2}(\nabla_\mu X_{\nu} + \nabla_\nu X_{\mu}).
\end{equation}
Thanks to \eqref{divergence of T}, we have
\begin{equation}\label{divergence of J}
 \nabla^\alpha J^{X}_{\alpha}[\phi] =  K^X [\phi] + \Phi \cdot X \phi
\end{equation}

In the null frame $\{e_1,e_2, e_3 =\Lb, e_4 =L\}$, we compute $\mathbb{T_{\alpha\beta}}[\phi]$ as
\begin{equation*}
\begin{split}
\mathbb{T}(L,L)[\phi] &= |L \phi|^2, \\
 \mathbb{T}(L,\Lb)[\phi] &=  |\nablaslash \phi|^2,\\
 \mathbb{T}(\Lb,\Lb)[\phi]&= |\Lb \phi|^2.
 \end{split}
\end{equation*}
This manifests the dominant energy condition for $\mathbb{T_{\alpha\beta}}[\phi]$.

We shall use $X = \Omega (\in \mathfrak{so}(3))$, $L$ and $\Lb$ as mutiplier vector fields, the corresponding deformation tensors and currents are computed as follows,
\begin{equation}\label{deformation tensor of Omega-L-Lb}
\begin{split}
  ^{(\Omega)}\pi_{\mu \nu} &= 0, \quad ^{(L)}\pi = \frac{2}{r} \gslash, \quad ^{(\Lb)}\pi = -\frac{2}{r} \gslash.\\
  K^\Omega &= 0, \quad  K^L = \frac{1}{r} L\phi \, \Lb \phi, \quad K^{\Lb} = -\frac{1}{r} L\phi \, \Lb \phi.
\end{split}
\end{equation}
where $\gslash$ is the restriction of the Minkowski metric $m$ on a two sphere $S_{\ub,u}$.

\begin{minipage}[!t]{0.3\textwidth}
\includegraphics[width=2.5in]{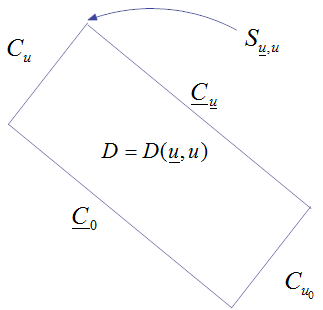}
\end{minipage}
\hspace{0.1\textwidth}
\begin{minipage}[!t]{0.5\textwidth}
We use $\mathcal{D}(u,\ub)$ to denote the space-time slab enclosed by the hypersurfaces $C_{u_0}$, $\Cb_0$, $C_{u}$ and $\Cb_{\ub}$. We integrate \eqref{divergence of J} on $\mathcal{D}(u,\ub)$, which is a domain enclosed by the null hypersurfaces $C_u, \Cb_{\ub}, C_{u_0}$ and $\Cb_0$, to derive
\begin{equation*}
\begin{split}
 &\quad \int_{C_u} \mathbb{T}[\phi](X,L)+\int_{\Cb_{\ub}} \mathbb{T}[\phi](X,\Lb)\\
 &=  \int_{C_{u_0}} \mathbb{T}[\phi](X,L)+\int_{\Cb_{0}} \mathbb{T}[\phi](X,\Lb)\\
 &\quad  +\doubleint_{\mathcal{D}(u,\ub)}  K^X [\phi] + \Phi \cdot X \phi.
 \end{split}
\end{equation*}
where  $L$ and $\Lb$ are corresponding normals of the null hypersurfaces $C_u$ and $\Cb_{\ub}$.
\end{minipage}

In applications, the data on $\Cb_{0}$ is always vanishing, thus, we have the following fundamental energy identity,
\begin{equation}\label{fundamental energy identity}
 \int_{C_u} \mathbb{T}[\phi](X,L)+\int_{\Cb_{\ub}} \mathbb{T}[\phi](X,\Lb)  =  \int_{C_{u_0}} \mathbb{T}[\phi](X,L) +\doubleint_{\mathcal{D}(u,\ub)}  K^X [\phi] + \Phi \cdot X \phi.
\end{equation}

\subsection{Null Forms}\label{section null form}
For a real valued quadratic form $Q$ defined on $\mathbb{R}^{3+1}$, it is called a \emph{null form} if for all null vector $\xi \in \mathbb{R}^{3+1}$, we have
$Q(\xi, \xi) = 0.$

We list seven obvious examples of null forms ($\alpha \neq \beta$),
\begin{equation}\label{basic null forms}
\begin{split}
 Q_0 (\xi, \eta) &= g(\xi, \eta), \\
 Q_{\alpha \beta}(\xi, \eta) &= \xi_\alpha \eta_\beta - \eta_\alpha \xi_\beta,
 \end{split}
\end{equation}
where $\xi, \eta \in \mathbb{R}^{3+1}$ and $\alpha, \beta \in \{0,1,2,3\}$. In fact, we can easily show that the space of all null forms on $\mathbb{R}^{3+1}$ is a real vector space and its dimension is 7. The above seven quadratic forms yield a basis for the space of null forms. Thus for an arbitrary null form $Q$, it can be written as a $\mathbb{R}$-linear combination of basic null forms in \eqref{basic null forms}.

Given two scalar function $\phi$, $\psi$ and a null form $Q(\xi,\eta) = Q^{\alpha\beta}\, \xi_\alpha \, \eta_\beta$, the expression $Q(\nabla \phi, \nabla \psi)$ means $Q(\nabla \phi, \nabla \psi) = Q^{\alpha\beta} \, \partial_\alpha \phi \, \partial_\beta \psi.$

For a given rotational Killing vector field $\Omega \in \mathfrak{so}(3)$, we have
\begin{equation}\label{commute null form with Omega}
 \Omega Q(\nabla \phi, \nabla \psi) = Q(\nabla \Omega \phi, \nabla \psi) + Q(\nabla \phi, \nabla \Omega \psi) + \widetilde{Q}(\nabla \phi, \nabla \psi),
\end{equation}
where $\widetilde{Q}$ is another null form. It suffices to check this claim on the basic null forms in \eqref{basic null forms}. In fact, for $\Omega = \Omega_{ij}$, one can check immediately that,
\begin{align*}
 \Omega_{ij} Q_0(\nabla \phi, \nabla \psi) &= Q_0(\nabla \Omega_{ij} \phi, \nabla \psi) + Q_0(\nabla \phi, \nabla \Omega_{ij} \psi),\\
 \Omega_{ij} Q_{\alpha\beta}(\nabla \phi, \nabla \psi) &= Q_{\alpha\beta}(\nabla \Omega_{ij} \phi, \nabla \psi) + Q_{\alpha\beta}(\nabla \phi, \nabla \Omega_{ij} \psi) + \widetilde{Q}(\nabla \phi, \nabla \psi),
\end{align*}
where
\begin{equation*}
 \widetilde{Q} = \delta_{i\alpha} Q_{j \beta} - \delta_{j \alpha} Q_{i \beta} +  \delta_{j \beta} Q_{i \alpha}-  \delta_{i\beta} Q_{j \alpha}.
\end{equation*}

For a vector field $X$, we denote
\begin{equation*}
     (Q \circ X)(\nabla \phi, \nabla \psi) = Q(\nabla X \phi, \nabla \psi) + Q(\nabla \phi, \nabla X \psi),
\end{equation*}
and $[Q, X] = X Q- Q\circ X$, we then have
\begin{equation*}
\begin{split}
  [L, Q_0](\nabla \phi, \nabla \psi) & =\frac{2}{r}(Q_0(\nabla \phi, \nabla \psi)+L\phi \Lb\psi+\Lb\phi L\psi), \\
   [\Lb, Q_0](\nabla \phi, \nabla \psi) & =-\frac{2}{r}(Q_0(\nabla \phi, \nabla \psi)+L\phi \Lb\psi+\Lb\phi L\psi),\\
   [L, Q_{ij}](\nabla \phi, \nabla \psi) & =\frac{2}{r}Q_{ij}(\nabla \phi, \nabla \psi)+\frac{1}{2r^2}\{(\Lb\phi-L\phi)\Omega_{ij}\psi+(L\psi-\Lb\psi)\Omega_{ij}\phi\}, \\
   [\Lb, Q_{ij}](\nabla \phi, \nabla \psi) & =-\frac{2}{r}Q_{ij}(\nabla \phi, \nabla \psi)+\frac{1}{2r^2}\{(L\phi-\Lb\phi)\Omega_{ij}\psi+(\Lb\psi-L\psi)\Omega_{ij}\phi\},\\
   [L, Q_{0i}](\nabla \phi, \nabla \psi) & =\frac{1}{r}Q_{0i}-\frac{x_i}{2r^2}(\Lb\phi L\psi-L\phi \Lb\psi), \\
   [\Lb, Q_{0i}](\nabla \phi, \nabla \psi) & =-\frac{1}{r}Q_{0i}-\frac{x_i}{2r^2}(L\phi \Lb\psi-\Lb\phi L\psi).
\end{split}
\end{equation*}
Schematically, we write the above as
\begin{equation}\label{commutates of L LB and Q}
\begin{split}
  [L, Q](\nabla \phi, \nabla \psi) &=\frac{1}{r}[Q(\nabla \phi, \nabla \psi)+L\phi \Lb\psi+\Lb\phi L\psi+L\phi \nablaslash \psi\\
  &\quad +L\psi \nablaslash \phi+\Lb\phi \nablaslash \psi+\Lb\psi \nablaslash \phi], \\
   [\Lb, Q](\nabla \phi, \nabla \psi) &=\frac{1}{r}[Q(\nabla \phi, \nabla \psi)+L\phi \Lb\psi+\Lb\phi L\psi+L\phi \nablaslash \psi\\
   &\quad +L\psi \nablaslash \phi+\Lb\phi \nablaslash \psi+\Lb\psi \nablaslash \phi].
\end{split}
\end{equation}

Besides the above algebraic properties of null forms, from the analytic point of view, in a null form $Q(\nabla \phi, \nabla \psi)$, a \emph{bad} component is always coupled to a \emph{good} component. To make a precise statement, we remark that in this
paper, all the derivatives of $\varphi$ involving the outgoing
direction $L$ (for example $L\varphi$ and $L\nablaslash \varphi$)
are bad since in $L^{\infty}$ norm, their size is comparable to
$\delta^{-\frac{1}{2}}$ which is \emph{large} and it
is in this sense they are bad components; for other derivatives,
their size is at least as good as $\delta^{\frac{1}{4}}$ which is
\emph{small} and it is in this sense they are good components. To
see why a bad component is coupled to a good component, we use null
frame $\{e_1,e_2, e_3 =\Lb, e_4 =L\}$ to express the null form as
follows,
\begin{align}\label{null decomposition of null forms}
 Q(\nabla \phi, \nabla \psi) &= Q^{43} \, L \phi \, \Lb\psi + Q^{34}\, \Lb \phi \, L\psi + Q^{4 a}\, L\phi \, \nablaslash_i \psi + Q^{3 a}\, \Lb\phi \, \nablaslash_a \psi \\
\notag &\quad + Q^{a 4}\, \nablaslash_a \phi \, L\psi + Q^{a 4}\, \nablaslash_a \phi \, \Lb\psi + Q^{ab}\, \nablaslash_a \phi \, \nablaslash_b \psi.
\end{align}
Once again, to prove \eqref{null decomposition of null forms}, it suffices to check for basic null forms in \eqref{basic null forms}.

In particular, \eqref{null decomposition of null forms} shows that $L\phi \cdot L \psi$ is forbidden which is a product of two bad components. We also observe that the coefficients appearing in \eqref{null decomposition of null forms} are bounded by a universal constant. Therefore, in applications, we shall bound the null form pointwisely as follows,
\begin{align}\label{null form bound}
 |Q(\nabla \phi, \nabla \psi)| &\lesssim |L \phi| \,|\Lb\psi| + |\Lb \phi|\,|L\psi| + |\nablaslash \phi|\, |\nablaslash \psi| \\
\notag &\quad +(| L\phi| + |\Lb \phi|)|\nablaslash \psi|  + |\nablaslash \phi|(|L\psi| +  |\Lb\psi|) .
\end{align}

\subsection{Sobolev and Gronwall's Inequalities} \label{Soboleve and Gronwall}
We first recall Sobolev inequalities on $C_u$, $\Cb_{\ub}$ and $S_{\ub, u}$.
 For any real valued function $\phi$ and $q \geq -\frac{1}{2}$, we have
 \begin{equation*}
 |u|^{\frac{1}{2}} \|\phi \|_{L^{4}(S_{\ub,u})} \lesssim \| L\phi \|^{\frac{1}{2}}_{L^{2}(C_{u})}(\| \phi \|^{\frac{1}{2}}_{L^{2}(C_{u})}
 +|u|^{\frac{1}{2}}\|\nablaslash\phi\|^{\frac{1}{2}}_{L^{2}(C_{u})}),
 \end{equation*}
\begin{equation*}
 |u|^{q}\|\phi\|_{L^{4}(S_{\ub,u})} \lesssim |u_{0}|^{q}\|\phi\|_{L^{4}(S_{\ub,u_{0}})}+
 \| |u|^{q}\Lb \phi \|^{\frac{1}{2}}_{L^{2}(\Cb_{\ub})}(\| |u|^{q-1}\phi \|^{\frac{1}{2}}_{L^{2}(\Cb_{\ub})}
 +\| |u|^{q}\nablaslash\phi \|^{\frac{1}{2}}_{L^{2}(\Cb_{\ub})}),
 \end{equation*}
\begin{equation*}
 \|\phi\|_{L^\infty(S_{\ub, u})} \lesssim |u|^{-\frac{1}{2}}\| \phi \|_{L^{4}(S_{\ub,u})} + |u|^{\frac{1}{2}}\| \nablaslash\phi \|_{L^{4}(S_{\ub,u})}.
 \end{equation*}
The proof is based on the standard isoperimeric inequality on unit sphere. We refer the reader to \cite{Ch-08} for a proof. We remark that in the first inequality, we assume that $\phi = 0$ on $\Cb_0$. This assumption is always valid when we apply the inequality in this paper. We also need a variant of the second inequality. In fact, Take $q=0,$ we derive
\begin{equation}\label{Sobolev Lb}
 \|\phi\|_{L^{4}(S_{\ub,u})} \lesssim \|\phi\|_{L^{4}(S_{\ub,u_{0}})}+
 \| \Lb \phi \|^{\frac{1}{2}}_{L^{2}(\Cb_{\ub})}(\| |u|^{-1}\phi \|^{\frac{1}{2}}_{L^{2}(\Cb_{\ub})}
 +\| \nablaslash\phi \|^{\frac{1}{2}}_{L^{2}(\Cb_{\ub})}).
 \end{equation}

Next, we recall the standard Gronwall's inequality. Let $\phi(t)$ be a non-negative function defined on an interval $I$ with
initial point $t_0$. If $\phi$ satisfies the following ordinary differential inequality,
\begin{equation*}
 \frac{d}{d t} \phi \leq a \cdot \phi + b,
\end{equation*}
where two non-negative functions  $a$, $b \in L^1(I)$, then for all $t \in I$, we have the following estimates,
\begin{equation*}
 \phi(t) \leq e^{A(t)}(\phi(t_0)+\int_{t_0}^t e^{-A(\tau)}b(\tau)d\tau),
\end{equation*}
where $A(t) = \int_{t_0}^t a(\tau) d\tau$. The proof is
straightforward. And there is another version of Gronwall's inequality \cite{K-R-09}, which will be useful in the proof.
\begin{lemma}
Let $f(x, y), g(x, y)$ be positive functions defined in the rectangle, $0\leq x\leq x_0, 0\leq y\leq y_0$ which verify the inequality,
\begin{equation*}
f(x, y)+g(x, y)\lesssim J +a\int_{0}^{x}f(x', y)dx'+b \int_{0}^{y}g(x,y')dy'
\end{equation*}
for some nonnegative constants $a, b$ and $J.$ Then, for all $0\leq x\leq x_0, 0\leq y\leq y_0$,
\begin{equation*}
f(x, y), g(x, y)\lesssim Je^{ax+by}.
\end{equation*}
\end{lemma}

\section{Initial Data for Region 1 and 2}\label{Section Initial Data in Short Pulse Regime}

Let $\widetilde{C}_{u_0}$ be a truncated light-cone defined by
\begin{equation*}
\widetilde{C}_{u_0} = \{p \in \mathbb{R}^{3+1} | u(p) = u_0,  u_0 \leq \ub(p) \leq \delta \},
\end{equation*}
and $C_{u_0}^{[0,\delta]}$ be a truncated light-cone defined by
\begin{equation*}
C_{u_0}^{[0,\delta]} = \{p \in \mathbb{R}^{3+1} | u(p) = u_0,  0 \leq \ub(p) \leq \delta \}.
\end{equation*}

First of all, we require that the data of \eqref{Main Equation} is trivial on $\widetilde{C}_{u_0} - C_{u_0}^{[0,\delta]}$, i.e.,
\begin{equation*}
 \varphi(x) \equiv 0, \quad \text{for all} \,\,\,\, x \in \widetilde{C}_{u_0} - C_{u_0}^{[0,\delta]}.
\end{equation*}
Therefore, according to the weak Huygens principle, the solution of \eqref{Main Equation} is zero in Region 1, i.e. the future domain of dependenc of $\widetilde{C}_{u_0} - C_{u_0}^{[0,\delta]}$. That is, $\varphi(x) \equiv 0$ if $ \ub(x) \leq 0$ and $u(x) \geq u_0$. In particular, $\varphi \equiv 0$ on $\Cb_0$ up to infinite order.\\

Secondly, we prescribe $\varphi$ on $C_{u_0}^{[0,\delta]}$ by
\begin{equation}\label{precise short pulse data on C_0}
 \varphi(\ub, u_0, \theta) = \frac{\delta^{\frac{1}{2}}}{|u_0|}\psi_0 (\frac{\ub}{\delta}, \theta),
\end{equation}
where $\psi_0: (0,1) \times \mathbb{S}^2 \rightarrow \mathbb{R}$ is a \emph{fixed} compactly supported smooth function with $L^2$ norm approximately $E_0$. The factor $\dfrac{1}{u_0}$ is natural because it manifests the correct decay for free waves.

The data given in the above form is called a \emph{short pulse data}, a name invented by Christodoulou in \cite{Ch-08}. In his work, he prescribes the shear (more precisely, the conformal geometry) of the initial null hypersurface in a similar form as \eqref{precise short pulse data on C_0}. The shear in the situation of \cite{Ch-08} is exactly the initial data for Einstein vacuum equation.

We remark that the above data is not a small data in the following sense: the derivative of the data can be extremely \emph{large} if $\delta$ is sufficiently small. In fact, this can be easily observed once we take $\dfrac{\partial}{\partial \ub}$ derivatives. We also remark that the energy flux of the data is approximately $E_0$ on $C_{u_0}^{[0,\delta]}$ which is far away from $0$.

In sequel, for most of the computations, we need commutator formulas and we collect all of them here. Denote Lie derivatives $D=\mathcal{L}_L$ and $\underline{D}=\mathcal{L}_{\underline{L}}$, then
\begin{equation}\label{commutates}
\begin{split}
[\mathcal{L}_\Omega, \nablaslash]&=0, \quad [D, \nablaslash]=0, \quad [\underline{D}, \nablaslash]=0,\\
[\Box, \Omega]&=0, \quad [D, \Omega]=0, \quad [\underline{D}, \Omega]=0,\\
[\Box, L]&=\frac{1}{r^2}(L-\Lb)+\frac{2}{r}\laplacianslash, \,\,\,[\Box, \Lb]=\frac{1}{r^2}(\Lb-L)-\frac{2}{r}\laplacianslash.
\end{split}
\end{equation}
If we commute $\Omega$ with \eqref{Main Equation} $n$ times, with respect to \eqref{commutates}, we have\footnote{ We shall ignore the numerical constants since they are irrelevant in the context.}
\begin{equation}\label{commute n Omega with main equation}
 \Box \Omega^{n} \varphi = \sum_{ p+q \leq n} Q(\nabla\Omega^{p}\phi, \nabla\Omega^{q}\phi),
\end{equation}
and these $Q$'s may be different. W commute $L, \Omega$ with \eqref{Main Equation} $n$ times, with respect to \eqref{commutates}, to derive
\begin{equation}\label{commute L Omega n with main equation}
\begin{split}
 \Box L \Omega^{n} \varphi &=  \sum_{p+q \leq n} Q(\nabla L\Omega^{p}\phi,
\nabla\Omega^{q}\phi) +\sum_{p+q \leq n}[L, Q](\nabla\Omega^{p}\phi, \nabla\Omega^{q}\phi) \\
 &\quad - \frac{1}{r^2}(\Lb \Omega^{n} \varphi -L \Omega^{n} \varphi)+\frac{2}{r}\laplacianslash \Omega^{n} \varphi.
\end{split}
\end{equation}
We commute $\Lb, \Omega$ with \eqref{Main Equation} $n$ times, with respect to \eqref{commutates}, to derive
\begin{equation}\label{commute Lb Omega n with main equation}
\begin{split}
 \Box \Lb \Omega^{n} \varphi &=  \sum_{p+q \leq n} Q(\nabla \Lb\Omega^{p}\phi,
\nabla\Omega^{q}\phi) +\sum_{p+q \leq n}[\Lb, Q](\nabla\Omega^{p}\phi, \nabla\Omega^{q}\phi) \\
 &\quad +\frac{1}{r^2}(\Lb \Omega^{n} \varphi -L \Omega^{n} \varphi)-\frac{2}{r}\laplacianslash \Omega^{n} \varphi.
\end{split}
\end{equation}
We remark, thanks to \eqref{commutates of L LB and Q}, we have the following pointwise estimates which gains a factor in $u$,
\begin{equation}\label{quadratic form bound}
\begin{split}
 &\quad |[L, Q](\nabla \phi, \nabla \psi)| + |[\Lb, Q](\nabla \phi, \nabla \psi)| \\
 &\lesssim \frac{1}{|u|}(|\nablaslash \phi||L \psi|+|\nablaslash \psi||L \phi| + |\nablaslash \phi||\Lb \psi| + |\nablaslash \psi||\Lb \phi|+ |L \phi||\Lb\psi|+ |L \psi||\Lb\phi|).
 \end{split}
\end{equation}

We now derive some preliminary estimates for the data on $C_{u_0}$. In view of \eqref{precise short pulse data on C_0}, by taking derivatives in $L$ or $\nablaslash$ direction, we have
\begin{equation*}
\begin{split}
 \|L\varphi\|_{L^{\infty}(C_{u_{0}})} &\leq \delta^{-\frac{1}{2}}u_{0}^{-1},\\
 \|\nablaslash\varphi\|_{L^{\infty}(C_{u_{0}})} &\leq \delta^{\frac{1}{2}} u_{0}^{-2}.
\end{split}
\end{equation*}
In fact, by taking $L$ or $\nablaslash$ derivatives consecutively, for $k \in \mathbb{Z}_{\geq 0}$, we obtain immediately,
\begin{equation}\label{L infinity on C0 for L or nablaslash directions}
\begin{split}
\|L\nablaslash^{k}\varphi\|_{L^{\infty}(C_{u_{0}})}
 &\lesssim_k \delta^{-\frac{1}{2}}u_{0}^{-k-1},\\
\|\nablaslash^{k+1}\varphi\|_{L^{\infty}(C_{u_{0}})} &\lesssim_k \delta^{\frac{1}{2}}u_{0}^{-k-2},\\
\|L^{2}\nablaslash^{k}\varphi\|_{L^{\infty}(C_{u_{0}})} &\lesssim_k \delta^{-\frac{3}{2}}u_{0}^{-k-1}.
\end{split}
\end{equation}
With the assistance of the original equation \eqref{Main Equation}, one can further derive the $L^\infty$ estimates for derivatives of $\varphi$ involving $\Lb$ directions. For this purpose, we first rewrite \eqref{Main Equation} in terms of null frame, namely,
\begin{equation}\label{Main Equation in null frame}\
\begin{split}
&\quad -L \Lb \varphi+\laplacianslash \varphi+\frac{1}{r}(L\varphi-\Lb\varphi)\\
&=2 Q^{34}\Lb\varphi  L\varphi  + 2 Q^{3a}\Lb \varphi \nablaslash_{a}\varphi + 2 Q^{4a} L\varphi  \nablaslash_{a}\varphi +Q^{ab}\nablaslash_{a} \varphi  \nablaslash_{b}\varphi.
\end{split}
\end{equation}

To estimate $\Lb \varphi$, we observe that \eqref{Main Equation in null frame} can be written as an ODE for $\Lb \varphi$ as follows, \footnote{ Since the exact numerical constants are irrelevant, we shall ignore the constants appearing in the coefficients.}
\begin{equation*}
 L (\Lb \varphi) = a \cdot \Lb \varphi+ b,
\end{equation*}
where
\begin{equation*}
\begin{split}
 a &= -(\frac{1}{r}+2Q^{34} L\varphi+2Q^{3a}\nablaslash_{a}\varphi),\\
  b  &= \frac{1}{r}L\varphi +\laplacianslash \varphi-2Q^{4a} L\varphi\nablaslash_{a}\varphi-Q^{ab} \nablaslash_{a} \varphi \nablaslash_{b}\varphi.
\end{split}
\end{equation*}
According to \eqref{L infinity on C0 for L or nablaslash directions}, we have,
\begin{equation*}
\begin{split}
\|a\|_{L^{\infty}(C_{u_{0}})}  &\lesssim \delta^{-\frac{1}{2}}u_{0}^{-1}, \\
\|b\|_{L^{\infty}(C_{u_{0}})}  &\lesssim \delta^{-\frac{1}{2}}u_{0}^{-2}.
\end{split}
\end{equation*}
We also have
\begin{equation*}
L |\Lb \varphi| \leq | L (\Lb \varphi)| \leq |a| \cdot |\Lb \varphi|+ |b|.
\end{equation*}
In view of the fact that $\Lb \varphi \equiv 0$ when $\ub =0$, by Gronwall's inequality (see Section \ref{Soboleve and Gronwall}), we have
\begin{equation*}
 |\Lb \varphi (u)| \lesssim \int_0^\delta e^{-A(\tau)} b(\tau)d\tau \lesssim \delta^{\frac{1}{2}}u_0^{-2}.
\end{equation*}
Hence,
 \begin{equation}\label{L infinity on Lb phi on C_0}
 \|\Lb\varphi\|_{L^{\infty}(C_{u_{0}})} \lesssim \delta^{\frac{1}{2}} u_{0}^{-2}.
 \end{equation}

To estimate $\Lb \nablaslash \varphi$, we first commute \eqref{Main Equation} with $\Omega$, that is, taking $n=1$ in \eqref{commute n Omega with main equation}. In null frame, we rewrite the equation as
 \begin{equation*}
-L\Lb\Omega\varphi+\laplacianslash\Omega \varphi +\frac{1}{r}(L\Omega\varphi-\Lb\Omega\varphi) = 2 Q_{1}(\nabla\Omega\varphi, \nabla\varphi)+ Q_{2}(\nabla\phi, \nabla\phi).
 \end{equation*}
We then proceed in a similar way as above, by Gronwall's inequality, we obtain
\begin{equation*}
\|\Lb\Omega\varphi\|_{L^{\infty}(C_{u_{0}})}\lesssim \delta^{\frac{1}{2}}u_{0}^{-2}.
\end{equation*}
Therefore, according to \eqref{Compare Omega and nablaslash}, we have
\begin{equation}\label{L infinity on Lb nablaslash phi on C_0}
\|\Lb\nablaslash\varphi\|_{L^{\infty}(C_{u_{0}})}\lesssim \delta^{\frac{1}{2}}u_{0}^{-3}.
 \end{equation}

Similarly, we can commute \eqref{Main Equation} with two and three $\Omega$'s and obtain
\begin{equation}\label{L infinity on Lb nablaslash nablaslash phi on C_0}
\begin{split}
\|\Lb\nablaslash^2 \varphi\|_{L^{\infty}(C_{u_{0}})}&\lesssim \delta^{\frac{1}{2}}u_{0}^{-4},\\
\|\Lb\nablaslash^3 \varphi\|_{L^{\infty}(C_{u_{0}})}&\lesssim \delta^{\frac{1}{2}}u_{0}^{-5}.
\end{split}
 \end{equation}

\begin{remark}[Key: Relaxation of the Estimates]
To obtain existence theorems for \eqref{Main Equation}, we have to derive certain estimates on $\varphi$ (as well as on its derivatives). Those estimates must be
valid on the initial hypersurface and they should propagate along the evolution to later null hypersurfaces. For this purpose, we shall use a slightly weaker version of estimates for $\nablaslash^k \varphi$ than those in \eqref{L infinity on C0 for L or nablaslash directions}, namely, the following estimates
\begin{equation*}
  \|\nablaslash^{k+1}\varphi\|_{L^{\infty}(C_{u_{0}})} \lesssim_k u_{0}^{-k-\frac{3}{2}}.
\end{equation*}
One expects it should be easier to prove the relaxed estimates propagating along the flow of \eqref{Main Equation} than the original ones in \eqref{L infinity on C0 for L or nablaslash directions}. This is the precisely relaxation of the propagation estimates mentioned in the introduction.
\end{remark}

To summarize, on the initial null hypersurface $C_{u_0}$, under the form of the short pulse data \eqref{precise short pulse data on C_0}, up to
fourth derivatives of $\varphi$ (this is the minimal number of derivatives we need to proceed an bootstrap argument, see next section), we have the following relaxed $L^\infty$ estimates,
\begin{equation}\label{L infinity on initial surface 1}
\begin{split}
 \|L\nablaslash^{k} \varphi\|_{L^{\infty}(C_{u_{0}})}  &\lesssim \delta^{-\frac{1}{2}}u_{0}^{-k-1}, \\
 \|\nablaslash^{k+1} \varphi\|_{L^{\infty}(C_{u_{0}})} &\lesssim u_{0}^{-\frac{3}{2}-k}, \\
 \|\Lb\nablaslash^{k}\varphi\|_{L^{\infty}(C_{u_{0}})}  &\lesssim \delta^{\frac{1}{2}}u_{0}^{-2-k}.
\end{split}
\end{equation}
for $k=0, 1, 2, 3,$ as well as
\begin{equation}\label{L infinity on initial surface 2}
\begin{split}
 \|L^{2}\varphi\|_{L^{\infty}(C_{u_{0}})} &\lesssim \delta^{-\frac{3}{2}}u_{0}^{-1}, \\
 \|L^{2}\nablaslash \varphi\|_{L^{\infty}(C_{u_{0}})} &\lesssim \delta^{-\frac{3}{2}}u_{0}^{-2},\\
  \|L^{2}\nablaslash^{2} \varphi\|_{L^{\infty}(C_{u_{0}})} &\lesssim \delta^{-\frac{3}{2}}u_{0}^{-3}.
\end{split}
\end{equation}

For wave equations, we expect the the information propagating along evolution should be more or less contained the energies of the solutions, i.e. $L^2$ norms of derivatives of $\varphi$. This heuristic leads us to consider the $L^2$ norms of the data on $C_{u_0}$.

According to the $L^\infty$ estimates in \eqref{L infinity on initial surface 1} and \eqref{L infinity on initial surface 2}, we obtain immediately the following $L^2$ estimates (oberve that the area of $C_{u_0}$ is comparable to $\delta \cdot u_0^2$),
\begin{equation}\label{L2 on initial surface 1}
\begin{split}
 \|L\nablaslash^{k} \varphi\|_{L^2(C_{u_{0}})}  &\lesssim u_{0}^{-k}, \\
 \|\nablaslash^{k+1} \varphi\|_{L^{2}(C_{u_{0}})}  &\lesssim \delta^{\frac{1}{2}}u_{0}^{-\frac{1}{2}-k},
 \end{split}
\end{equation}
for $k=0, 1, 2, 3,$ and
\begin{equation}\label{L2 on initial surface 2}
\begin{split}
 \|L^{2}\varphi\|_{L^{\infty}(C_{u_{0}})} &\lesssim \delta^{-1}, \\
 \|L^{2}\nablaslash \varphi\|_{L^{2}(C_{u_{0}})} &\lesssim \delta^{-1}u_{0}^{-1},\\
 \|L^{2}\nablaslash^{2} \varphi\|_{L^{2}(C_{u_{0}})} &\lesssim \delta^{-1}u_{0}^{-2}.
\end{split}
\end{equation}
We the remark that those $L^2$ estimates are also relaxed estimates. In next section, we shall show that, up to a universal constant, the estimates in \eqref{L2 on initial surface 1} and \eqref{L2 on initial surface 2} (the parameter $u_0$ will be replaced by $u$) will hold on all later outgoing null hypersurfaces $C_u$ where $u_0 \leq u \leq -1$ provided that the solution of \eqref{Main Equation} can be constructed up to $C_u$.

\section{An A priori Estimates up to Fourth Derivatives}\label{Section An A priori Estimates up to Third Derivatives}

This section is the technical heart of the paper. We assume that there exists a solution of \eqref{Main Equation} defined on the domain $\mathcal{D}_{u,\ub}$ which is enclosed by the null hypersurfaces $C_{u}$, $\Cb_{\ub}$, $C_{u_0}$ and $\Cb_0$. The goal is to show that estimates \eqref{L2 on initial surface 1} and \eqref{L2 on initial surface 2}, which are valid on $C_{u_0}$, also hold on $C_u$.\\

We slightly abuse the notations: we use $C_u$ to denote $C^{[0,\ub]}_u$ (i.e. $\ub' \in [0, \ub]$)and $\Cb_{\ub}$ to denote $\Cb^{[u_0,u]}_{\ub}$. We now define a family of energy norms as follows,
\begin{equation}
\begin{split}
E_1(u,\ub) &= \|L \varphi\|_{L^2(C_u)} + \delta^{-\frac{1}{2}} |u|^{\frac{1}{2}}\|\nablaslash \varphi\|_{L^2(C_u)},\\
\Eb_1(u,\ub) &=\|\nablaslash \varphi\|_{L^2(\Cb_{\ub})} + \delta^{-\frac{1}{2}} |u|^{\frac{1}{2}}\|\Lb \varphi\|_{L^2(\Cb_{\ub})},\\
E_2(u,\ub) &=|u|\|L \nablaslash \varphi\|_{L^2(C_u)} + \delta^{-\frac{1}{2}} |u|^{\frac{3}{2}}\|\nablaslash^2 \varphi\|_{L^2(C_u)},\\
\Eb_2(u,\ub) &=  \||u|\nablaslash^2 \varphi\|_{L^2(\Cb_{\ub})} + \delta^{-\frac{1}{2}} |u|^{\frac{1}{2}}\||u|\Lb \nablaslash \varphi\|_{L^2(\Cb_{\ub})},\\
E_3(u,\ub) &= |u|^2\|L \nablaslash^2 \varphi\|_{L^2(C_u)} + \delta^{-\frac{1}{2}} |u|^{\frac{5}{2}}\|\nablaslash^3 \varphi\|_{L^2(C_u)},\\
\Eb_3(u,\ub) &= \||u|^2 \nablaslash^3 \varphi\|_{L^2(\Cb_{\ub})} + \delta^{-\frac{1}{2}} |u|^{\frac{1}{2}}\||u|^{2}\Lb \nablaslash^2 \varphi\|_{L^2(\Cb_{\ub})},\\
E_4(u,\ub) &=|u|^3\|L \nablaslash^3 \varphi\|_{L^2(C_u)} + \delta^{-\frac{1}{2}} |u|^{\frac{7}{2}}\|\nablaslash^3 \varphi\|_{L^2(C_u)},\\
\Eb_4(u, \ub) &= \||u|^3 \nablaslash^4 \varphi\|_{L^2(\Cb_{\ub})} + \delta^{-\frac{1}{2}} |u|^{\frac{1}{2}}\||u|^{3}\Lb \nablaslash^3 \varphi\|_{L^2(\Cb_{\ub})}.
\end{split}
\end{equation}
We also need another family of norms which involves at least two null derivatives. They are defined as follows,
\begin{equation}
\begin{split}
F_2(u,\ub) &=\delta\|L^2 \varphi\|_{L^2(C_u)}, \\
\Fb_2(u, \ub) &=|u|^{\frac{1}{2}} \|\Lb^2 \varphi\|_{L^2(\Cb_{\ub})},\\
F_3(u,\ub) &=\delta|u| \|L^2 \nablaslash \varphi\|_{L^2(C_u)},\\
\Fb_3(u,\ub) &=|u|^{\frac{1}{2}} \||u| \Lb^2 \nablaslash \varphi\|_{L^2(\Cb_{\ub})},\\
F_4(u,\ub) &=\delta|u|^{2} \|L^2 \nablaslash^{2} \varphi\|_{L^2(C_u)}, \\
\Fb_4(u,\ub) &=|u|^{\frac{1}{2}} \||u|^{2} \Lb^2 \nablaslash^{2} \varphi\|_{L^2(\Cb_{\ub})}.
\end{split}
\end{equation}
We shall prove the following propagation estimates,
\begin{Main A priori Estimates}\label{a priori estimates} If $\delta$ is sufficiently small, for all initial data of \eqref{Main Equation} and all $I_4 \in \mathbb{R}_{>0}$ which satisfy
\begin{equation}\label{initial bound}
\begin{split}
&\quad E_1(u_0,\delta) + E_2(u_0,\delta)+E_3(u_0,\delta)+E_4(u_0,\delta) \\
&+ F_2(u_0,\delta) + F_3(u_0,\delta)+F_4(u_0,\delta) \leq I_4,
\end{split}
\end{equation}
there is a constant $C(I_4)$ depending only on $I_4$ (in particular, not on $\delta$ and $u_0$), so that
\begin{equation}\label{main estimates}
\sum_{i=1}^4 [E_i(u,\ub) + \Eb_i(u,\ub)]+\sum_{j=2}^4 [F_j(u,\ub) + \Fb_j(u,\ub)] \leq C(I_4).
\end{equation}
The subindex $4$ in $I_4$ denotes the number of derivatives used in the energy norms.
\end{Main A priori Estimates}

\subsection{Bootstrap Argument}
We will perform a standard bootstrap argument to prove the \textbf{Main A priori Estimates}. We assume that
\begin{equation}\label{bootstrap assumption}
\sum_{i=1}^4 [E_i(u',\ub') + \Eb_i(u',\ub')]+\sum_{j=2}^4 [F_j(u',\ub') + \Fb_j(u',\ub')] \leq M,
\end{equation}
for all $u' \in [u_0, u]$ and $\ub' \in [0, \ub]$, where $M$ is a sufficiently large constant. Since we have assumed the existence of the solution up to  $C_{u}$ and $\Cb_{\ub}$, we can always choose such a $M$ which may depend on the $\varphi$. At the end of the current section, we will show that we can actually choose $M$ in such a way that it depends only on the norm of the initial data but not the profile $\varphi$. This will yield the \textbf{Main A priori Estimates}.

\subsection{Preliminary Estimates}\label{section Preliminary Estimates}
Under the bootstrap assumption \eqref{bootstrap assumption}, we first derive $L^\infty$ for one derivatives of $\varphi$. As a byproduct, we will also obtain the $L^4$ estimates for two derivatives of $\varphi$. For this purpose, we shall repeatedly use the Sobolev inequalities stated in Section \ref{Soboleve and Gronwall}.

We start with $L \varphi$. According to Sobolev inequalities, we have
\begin{align*}
 |u|^{\frac{1}{2}}\| L\varphi \|_{L^{4}(S_{\ub,u})} &\lesssim  \| L^{2}\varphi\|^{\frac{1}{2}}_{L^{2}(C_{u})}(\| L\varphi \|^{\frac{1}{2}}_{L^{2}(C_{u})}
 +|u|^{\frac{1}{2}}\| L\nablaslash\varphi \|_{L^{2}(C_{u})}^{\frac{1}{2}})\\
 &\lesssim (\delta^{-1} M )^\frac{1}{2}(M^\frac{1}{2}+ |u|^\frac{1}{2}(|u|^{-1}M)^\frac{1}{2}).
 \end{align*}
Hence,
\begin{equation}\label{e 1}
\| L\varphi \|_{L^{4}(S_{\ub,u})}\lesssim \delta^{-\frac{1}{2}}|u|^{-\frac{1}{2}} M.
\end{equation}
Similarly, we have
\begin{align*}
 |u|^{\frac{1}{2}}\| L \nablaslash \varphi \|_{L^{4}(S_{\ub,u})} &\lesssim  \| L^{2}\nablaslash\varphi\|^{\frac{1}{2}}_{L^{2}(C_{u})}(\| L \nablaslash\varphi \|^{\frac{1}{2}}_{L^{2}(C_{u})}
 +|u|^{\frac{1}{2}}\| L{\nablaslash}^2\varphi \|_{L^{2}(C_{u})}^{\frac{1}{2}})\\
 &\lesssim (\delta^{-1} |u|^{-1} M )^\frac{1}{2}((|u|^{-1} M)^\frac{1}{2}+ |u|^\frac{1}{2}(|u|^{-2}M)^\frac{1}{2}).
\end{align*}
Thus,
\begin{equation}\label{e_2}
\| L\nablaslash \varphi \|_{L^{4}(S_{\ub,u})} \lesssim \delta^{-\frac{1}{2}}|u|^{-\frac{3}{2}}M.
\end{equation}
Combine \eqref{e 1} and \eqref{e_2}, we have
\begin{equation}\label{e_3}
\begin{split}
 \|L\varphi\|_{L^\infty}&\lesssim  |u|^{-\frac{1}{2}}\| L\varphi \|_{L^{4}(S_{\ub,u})}
 +|u|^{\frac{1}{2}}\| L\nablaslash\varphi \|_{L^{4}(S_{\ub,u})}\\
 &\lesssim \delta^{-\frac{1}{2}}|u|^{-1}M.
 \end{split}
 \end{equation}

We now treat $\nablaslash \varphi$. According to Sobolev inequalities, we have
\begin{equation*}
\begin{split}
 |u|^{\frac{1}{2}}\| \nablaslash \varphi \|_{L^{4}(S_{\ub,u})} &\lesssim  \| L \nablaslash\varphi\|^{\frac{1}{2}}_{L^{2}(C_{u})}(\| \nablaslash \varphi \|^{\frac{1}{2}}_{L^{2}(C_{u})}
 +|u|^{\frac{1}{2}}\| \nablaslash^2 \varphi \|_{L^{2}(C_{u})}^{\frac{1}{2}})\\
 &\lesssim (|u|^{-1}M )^\frac{1}{2}((\delta^{\frac{1}{2}}|u|^{-\frac{1}{2}}M)^\frac{1}{2}+ |u|^\frac{1}{2}(\delta^{\frac{1}{2}} |u|^{-\frac{3}{2}}M)^\frac{1}{2}).
\end{split}
\end{equation*}
Thus,
\begin{equation}\label{e_4}
\| \nablaslash\varphi\|_{L^{4}(S_{\ub,u})}\lesssim \delta^{\frac{1}{4}}|u|^{-\frac{5}{4}}M.
\end{equation}
Similarly,
\begin{equation}\label{e_5}
\| \nablaslash^2 \varphi\|_{L^{4}(S_{\ub,u})}\lesssim \delta^{\frac{1}{4}}|u|^{-\frac{9}{4}}M.
\end{equation}
Combine \eqref{e_4} and \eqref{e_5}, we have
\begin{equation}\label{e_6}
\begin{split}
 \|\nablaslash \varphi\|_{L^\infty} &\lesssim  |u|^{-\frac{1}{2}}\| \nablaslash\varphi \|_{L^{4}(S_{\ub,u})}
 +|u|^{\frac{1}{2}}\| \nablaslash^2 \varphi \|_{L^{4}(S_{\ub,u})}\\
 &\lesssim \delta^{\frac{1}{4}}|u|^{-\frac{7}{4}}M.
 \end{split}
 \end{equation}

It remains to estimate $\Lb\varphi$. For $\|\Lb \varphi\|_{L^4}$, according to \eqref{Sobolev Lb}, we have
\begin{equation*}
 \|\Lb\varphi\|_{L^{4}(S_{\ub,u})} \lesssim \|\Lb\varphi\|_{L^{4}(S_{\ub,u_{0}})}+
 \|\Lb^2 \varphi \|_{L^{2}(\Cb_{\ub})}^{\frac{1}{2}}(\| |u|^{-1}\Lb\varphi \|_{L^{2}(\Cb_{\ub})}^{\frac{1}{2}}+\| \Lb\nablaslash\varphi \|_{L^{2}(\Cb_{\ub})}^{\frac{1}{2}}).
 \end{equation*}
Since first three terms are appearing in the bootstrap assumptions (or on the initial hypersurface $H_{u_0}$), we can control them exactly as before. For the last term, we can restrict the inequality on the part of $\Cb_{\ub}$ where the affine parameter $u'$ of $\Lb$ is in the interval $[u_0, u]$. Thus, we have
\begin{equation*}
\| |u'|^{-1}\Lb\varphi(u',\ub,\theta) \|_{L^{2}(\Cb_{\ub})} \leq |u|^{-1}\| \Lb\varphi \|_{L^{2}(\Cb_{\ub})}.
\end{equation*}
The righthand side is again a term in \eqref{bootstrap assumption}. This allows us to derive
\begin{equation}\label{e_7}
\| \Lb\varphi \|_{L^{4}(S_{\ub,u})} \lesssim \delta^{\frac{1}{4}}|u|^{-1}M.
\end{equation}
Similarly, we can derive
\begin{equation}\label{e_8}
\| \Lb\nablaslash\varphi \|_{L^{4}(S_{\ub,u})}\lesssim \delta^{\frac{1}{4}}|u|^{-2}M.
\end{equation}
We combine \eqref{e_7} and \eqref{e_8} to derive
\begin{equation}\label{e_9}
\| \Lb\varphi\|_{L^{\infty}} \lesssim \delta^{\frac{1}{4}}|u|^{-\frac{3}{2}}M.
\end{equation}
In the same way, we can derive $L^4, L^{\infty}$ estimates of two derivatives.
We summarize all the estimates in the following proposition.
\begin{proposition}\label{proposition L infinity and L4 estimates} Under the bootstrap assumption \eqref{bootstrap assumption}, we have
\begin{equation*}
\begin{split}
 &\quad \delta^{\frac{1}{2}}|u|\|L\varphi\|_{L^\infty} + \delta^{-\frac{1}{4}}|u|^{\frac{7}{4}}\|\nablaslash\varphi\|_{L^\infty} +  \delta^{-\frac{1}{4}}|u|^{\frac{3}{2}}\|\Lb\varphi\|_{L^\infty}\\
 &+\delta^{\frac{1}{2}}|u|^{2}\|L\nablaslash\varphi\|_{L^\infty} + \delta^{-\frac{1}{4}}|u|^{\frac{11}{4}}\|\nablaslash^{2}\varphi\|_{L^\infty} +  \delta^{-\frac{1}{4}}|u|^{\frac{5}{2}}\|\Lb\nablaslash\varphi\|_{L^\infty} \lesssim M\\
&\quad\delta^{\frac{1}{2}}|u|^{\frac{5}{2}}\| L\nablaslash^2 \varphi \|_{L^{4}(S_{\ub,u})} +\delta^{-\frac{1}{4}}|u|^{\frac{13}{4}} \|\nablaslash^3 \varphi\|_{L^{4}(S_{\ub,u})} + \delta^{-\frac{1}{4}}|u|^{3}\|\Lb\nablaslash^2\varphi \|_{L^{4}(S_{\ub,u})}\\
&+\delta^{\frac{1}{2}}|u|^{\frac{3}{2}}\| L\nablaslash \varphi \|_{L^{4}(S_{\ub,u})} +\delta^{-\frac{1}{4}}|u|^{\frac{9}{4}} \|\nablaslash^2 \varphi\|_{L^{4}(S_{\ub,u})} + \delta^{-\frac{1}{4}}|u|^{2}\|\Lb\nablaslash\varphi \|_{L^{4}(S_{\ub,u})}\\
&+\delta^{\frac{1}{2}}|u|^{\frac{1}{2}}\| L\varphi \|_{L^{4}(S_{\ub,u})} +\delta^{-\frac{1}{4}}|u|^{\frac{5}{4}} \|\nablaslash \varphi\|_{L^{4}(S_{\ub,u})} + \delta^{-\frac{1}{4}}|u|\|\Lb\varphi \|_{L^{4}(S_{\ub,u})}\lesssim M.
\end{split}
\end{equation*}
\end{proposition}

We observe that $L^\infty$ estimates on $\Lb \varphi$ (which of order $\delta^{\frac{1}{4}} |u|^{-\frac{3}{2}}$) is certainly worse than the initial estimates of $\Lb\varphi$ on $C_{u_0}$ (which is of order $\delta^{\frac{1}{2}} |u_0|^{-2}$). To rectify this loss, we derive a $L^2$ estimates of $\Lb \varphi$ on $C_u$ (instead of $\Cb_{\ub}$ appearing in the definition of $\Eb_1(u,\ub)$).

\begin{lemma}\label{lemma L2 of Lb phi}
 Under the bootstrap assumption \eqref{bootstrap assumption}, if $\delta^{\frac{1}{2}} M$  is sufficiently small, we have
 \begin{equation*}
\|\Lb\varphi\|_{L^2(C_u)} \lesssim \delta|u|^{-1}M.
\end{equation*}
\end{lemma}
\begin{proof} We multiply $\Lb\varphi$ on both side of the main equation \eqref{Main Equation in null frame} and integrate on $C_{u}$. In view of the fact that $\Lb\varphi \equiv 0$ on $S_{0, u}$ as well as \eqref{null form bound}, this leads to
 \begin{equation}\label{aa}
 \begin{split}
\int_{S_{\underline{u},u}}\!\!\!|\underline{L}\varphi|^{2}&\lesssim \int_{C_{u}^{\underline{u}}}\frac{1}{r}|L\varphi||\underline{L}\varphi|+ |\laplacianslash\varphi||\underline{L}\varphi|
+|\nablaslash\varphi||\underline{L}\varphi|^{2}\\
& +\int_{C_{u}^{\underline{u}}} |L\varphi||\underline{L}\varphi|^{2}+ |L\varphi| |\nablaslash\varphi||\underline{L}\varphi|+ |\nablaslash\varphi|^2|\underline{L}\varphi|,
\end{split}
\end{equation}
where the integral $ \int_{C_{u}^{\underline{u}}}$ means $\int_{0}^{\ub}\int_{S_{\ub',u}} d\ub'$. Let
$f^2(\ub) = \int_{C_{u}^{\underline{u}}}(\underline{L}\varphi)^{2}.$
We now estimate the terms at the right hand side of \eqref{aa} one by one. We have, for the first two terms
\begin{equation*}
 \begin{split}
\int_{C_{u}^{\underline{u}}}\frac{1}{r}|L\varphi||\underline{L}\varphi| &\lesssim |u|^{-1}f(\ub)M, \\
 \int_{C_{u}^{\underline{u}}}|\laplacianslash\varphi||\underline{L}\varphi| &\lesssim  \delta^{\frac{1}{2}}|u|^{-\frac{3}{2}}f(\ub)M.
 \end{split}
\end{equation*}
For the second two terms, we have
\begin{equation*}
 \begin{split}
\int_{C_{u}^{\underline{u}}} |\nablaslash\varphi||\underline{L}\varphi|^{2} &\lesssim \delta^{\frac{1}{4}}|u|^{-\frac{7}{4}}f^{2}(\ub)M, \\ \int_{C_{u}^{\underline{u}}}|L\varphi||\underline{L}\varphi|^{2}&\lesssim  \delta^{-\frac{1}{2}}|u|^{-1}f^2(\ub) M.
\end{split}
\end{equation*}
For the last two terms, we have
\begin{equation*}
\begin{split}
 \int_{C_{u}^{\underline{u}}}|L\varphi|| \nablaslash\varphi||\underline{L}\varphi| &\lesssim \delta^{\frac{1}{4}}|u|^{-\frac{7}{4}}f(\ub) M, \\ \int_{C_{u}^{\underline{u}}}|\nablaslash\varphi|^2|\underline{L}\varphi| &\lesssim \delta^{\frac{3}{4}}|u|^{-\frac{9}{4}}f(\ub) M.
 \end{split}
\end{equation*}
Back to \eqref{aa}, we have
\begin{equation*}
\frac{d}{d\underline{u}}f(\ub)^{2}\lesssim M(\delta^{-\frac{1}{2}}|u|^{-1}f(\ub)^{2}+|u|^{-1}f(\ub)),
 \end{equation*}
We then integrate on $C_{u}$ to derive
\begin{equation*}
f(\ub) \lesssim \frac{M}{|u|} \delta^{\frac{1}{2}}f(\ub)+\delta\frac{M}{|u|}.
 \end{equation*}
Since $\delta^{\frac{1}{2}} M$  is sufficiently small, this completes the proof
\end{proof}

\subsection{Energy Estimates for $E_{k}(u,\ub)$ and $\Eb_{k}(u,\ub)$  when $k \leq 3$}

Commute $\Omega^i$ with \eqref{Main Equation}, in view of \eqref{commute n Omega with main equation}. We use the scheme in Section \ref{Energy estimates scheme} for this equation where we take $\phi = \Omega^i \varphi, i=0, 1, 2$ and $X = L$. In view of \eqref{fundamental energy identity}, we have
\begin{equation}\label{ES-E1-E2-E3-a}
\begin{split}
\int_{C_{u}}|L \Omega^i \varphi|^{2}&+\int_{\underline{C}_{\underline{u}}}|\nablaslash \Omega^i \varphi|^{2} =\int_{C_{u_{0}}}|L\Omega^i \varphi|^{2}+\doubleint_{\mathcal{D}}Q(\nabla\Omega^{i}\phi, \nabla\phi)L \Omega^i \varphi \\
& + \doubleint_{\mathcal{D}}\sum_{p+q \leq i, p\neq i, q\neq i}
Q(\nabla\Omega^{p}\phi, \nabla\Omega^{q}\phi)L \Omega^i\varphi +\doubleint_{\mathcal{D}} \frac{1}{r} \Lb \Omega^i \varphi \cdot L \Omega^i\varphi .
\end{split}
\end{equation}
At this point, we rewrite the above equations as
\begin{equation*}
\begin{split}
\int_{C_{u}}|L \Omega^i \varphi|^{2}&+\int_{\underline{C}_{\underline{u}}}|\nablaslash \Omega^i \varphi|^{2} =\int_{C_{u_{0}}}|L\Omega^i \varphi|^{2}+ R + S +T.
\end{split}
\end{equation*}
where $R$, $S$ and $T$ are defined in an obvious way. Before deriving the estimates, we remark that for any function $\phi$, we actually have
\begin{equation*}
 \|\Omega^i\phi\|_{L^p(S_{\ub,u})} \sim \||u|^i|\nablaslash^i\phi|\|_{L^p(S_{\ub,u})},
\end{equation*}
which can be easily derived from \eqref{Compare Omega and nablaslash}.

Let us first consider $R$, i.e. the second integral terms on the right hand side of all equations \eqref{ES-E1-E2-E3-a}, in view of \eqref{null form bound}, it splits into
\begin{equation*}
\begin{split}
R_1 &= \doubleint_{\mathcal{D}}|\Lb\varphi||L\Omega^i \varphi|^2, \\
R_2 &= \doubleint_{\mathcal{D}}|\nablaslash \varphi||L\Omega^i \varphi|^2, \\
R_3 &= \doubleint_{\mathcal{D}}|L\varphi||\Lb\Omega^i\varphi||L\Omega^i \varphi|,\\
R_4 &= \doubleint_{\mathcal{D}}|L \varphi||\nablaslash\Omega^i\varphi||L\Omega^i \varphi|, \\
R_5 &= \doubleint_{\mathcal{D}}|\nablaslash\varphi||\Lb\Omega^i \varphi||L\Omega^i \varphi|,\\
R_6 &= \doubleint_{\mathcal{D}}(|\Lb\varphi|+|\nablaslash \varphi|)|\nablaslash\Omega^i\varphi||L\Omega^i \varphi|,
\end{split}
\end{equation*}
where $i=0,1, 2.$
Now we bound those terms one by one.

For $R_1$, we have
\begin{align*}
R_1 &\leq \int_{u_0}^{u}\|\Lb\varphi\|_{L^{\infty}} (\int_{C_{u'}} |L\Omega^i\varphi|^2 )du'\\
 &\lesssim \int_{u_0}^{u}\delta^{\frac{1}{4}}|u'|^{-\frac{3}{2}}M M^2 du' \\
 &\lesssim \delta^{\frac{1}{4}}|u|^{-\frac{1}{2}}M^3.
\end{align*}
For $R_2$, we bound $\nablaslash \varphi$ in $L^\infty$ then proceed exactly as above. This bound is better than $R_1$'s and we shall use the worse one,
\begin{equation*}
R_2 \lesssim \delta^{\frac{1}{4}}|u|^{-\frac{1}{2}}M^3.
\end{equation*}
For $R_3$, we have
\begin{align*}
 R_3 &\lesssim (\doubleint_{\mathcal{D}}|L\varphi|^2|L\Omega^i\varphi|^2)^{\frac{1}{2}}(\doubleint_{\mathcal{D}}|\Lb\Omega^i \varphi|^2)^{\frac{1}{2}} \\
&= (\int_{u_0}^{u}\|L\varphi\|^2_{L^{\infty}} \|L\Omega^i\varphi\|^2_{L^2(C_{u'})} du')^{\frac{1}{2}}(\int_{0}^{\ub}\||u|^i\Lb\nablaslash^i \varphi\|_{L^2(\Cb_{\ub'})}^2 d\ub')^{\frac{1}{2}}\\
&\lesssim \delta^{\frac{1}{2}}|u|^{-1}M^3.
\end{align*}
For $R_4$, since $i \leq 4$, we can use $L^4$ estimates in Proposition \ref{proposition L infinity and L4 estimates}. This is an easy but important observation since we are dealing with terms with lower (less that 4) order derivatives. And for the highest order derivative terms, we can not use this approach. We then have (note that we bound $\nablaslash\Omega^i\varphi$ by $L^4$ instead of $L^2$ to gain $\delta^\frac{1}{4}$.)
\begin{equation*}
\begin{split}
  R_4 &\leq \int_{u_0}^{u}\|L\varphi\|_{L^{2}(C_{u'})} \|\nablaslash\Omega^i\varphi\|_{L^4(C_{u'})}  \|L\Omega^i\varphi\|_{L^4(C_{u'})} du'\\
  &\lesssim \delta^{\frac{1}{4}}|u|^{-\frac{3}{4}}M^3.
  \end{split}
\end{equation*}
For $R_5$ and $R_6$, they can be bounded exactly in the same way as $R_4$, thus,
\begin{align*}
R_5  &\lesssim \delta^{\frac{1}{2}}|u|^{-\frac{5}{4}}M^3, \\
R_6  &\lesssim \delta^{\frac{1}{4}}|u|^{-\frac{3}{4}}M^3.
\end{align*}

We now move to $S$, i.e. the third integral terms on the right hand side of all equations \eqref{ES-E1-E2-E3-a}. A general form for the integrand can be written schematically as
\begin{equation*}
S = \doubleint_{\mathcal{D}}|\nabla \Omega^p \varphi||\nabla \Omega^q\varphi||L\Omega^i \varphi|.
\end{equation*}
Notice that at least one $\nabla$ in this formula is not $L$. Thus, we can estimate this term exact in the same way as we have done for $R_4$. This leads to
\begin{equation*}
S \lesssim\delta^{\frac{1}{4}}|u|^{-\frac{1}{2}}M^3.
\end{equation*}

Finally, for $T$, i.e. the last integral terms on the right hand side of all equations \eqref{ES-E1-E2-E3-a}. It is defined as
\begin{equation*}
T = \doubleint_{\mathcal{D}} \frac{1}{r}|\Lb\Omega^i\varphi||L\Omega^i \varphi|,
\end{equation*}
We bound it by
\begin{align*}
T &\lesssim \int^{u}_{u_{0}}\frac{1}{|u'|}\|L\Omega^i\varphi\|_{L^{2}(C_{u'})}\|\Lb\Omega^i\varphi\|_{L^{2}(C_{u'})}du'\\
&\lesssim \int^{u}_{u_{0}}\frac{1}{|u'|}M  \delta |u'|^{-1}Mdu' \\
&\lesssim \delta|u|^{-1}M^2
\end{align*}

Putting those estimates in \eqref{ES-E1-E2-E3-a}, in view of the size of initial data \eqref{L2 on initial surface 1} as well as $M \geq 1$, we obtain
\begin{equation*}
\int_{C_{u}}|L\Omega^i\varphi|^{2}+\int_{\underline{C}_{\underline{u}}}|\nablaslash\Omega^i\varphi|^{2} \lesssim I_4^2 + \delta^{\frac{1}{4}}|u|^{-\frac{1}{2}}M^3.
\end{equation*}
Hence,
\begin{equation}\label{ES-E1-E2-E3-b}
\|L\Omega^i\varphi\|_{L^2(C_u)} + \|\nablaslash\Omega^i \varphi\|_{L^2(\Cb_{\ub})} \lesssim I_4 + \delta^{\frac{1}{8}}|u|^{-\frac{1}{4}}M^{\frac{3}{2}},
\end{equation}
for $i=0,1,2$.\\

Next we switch $X$ to be $\Lb$. In view of \eqref{fundamental energy identity}, we have
\begin{equation}\label{ES-Eb1-Eb2-Eb3-a}
\begin{split}
&\int_{C_{u}}|\nablaslash \Omega^i \varphi|^{2}+\int_{\underline{C}_{\underline{u}}}|\Lb \Omega^i \varphi|^{2} =\int_{C_{u_{0}}}|\nablaslash \Omega^i \varphi|^{2}  + \doubleint_{\mathcal{D}}Q(\nabla \Omega^i \varphi, \nabla \varphi)\Lb \Omega^i \varphi  \\
&  + \doubleint_{\mathcal{D}}\sum_{p+q \leq i, p<i, q<i}
Q(\nabla\Omega^{p}\phi, \nabla\Omega^{q}\phi)\Lb \Omega^i \varphi -\doubleint_{\mathcal{D}} \frac{1}{r} \Lb \Omega^i \varphi \cdot L \Omega^i\varphi .
\end{split}
\end{equation}
for $i=0, 1$ and $2$.
At this point, we rewrite the above equations as
\begin{equation*}
\begin{split}
\int_{C_{u}}|\nablaslash \Omega^i \varphi|^{2}+\int_{\underline{C}_{\underline{u}}}|\Lb \Omega^i \varphi|^{2} =\int_{C_{u_{0}}}|\nablaslash \Omega^i \varphi|^{2}  +  R + S +T.
\end{split}
\end{equation*}
where $R$, $S$ and $T$ are defined in an obvious way.

We start with $R$, i.e. the second integral terms on the right hand side of all equations \eqref{ES-Eb1-Eb2-Eb3-a}, in view of \eqref{null form bound}, it splits into
\begin{equation*}
\begin{split}
 R_1 &= \doubleint_{\mathcal{D}}(|L \varphi|+|\nablaslash \varphi|)|\Lb\Omega^i \varphi|^2, \\
 R_2 &= \doubleint_{\mathcal{D}}(|\nablaslash\varphi|+|\Lb\varphi|)|L\Omega^i\varphi||\Lb\Omega^i \varphi|,\\
 R_3 &= \doubleint_{\mathcal{D}}(|\Lb \varphi|+|\nablaslash \varphi|)|\nablaslash\Omega^i\varphi||\Lb\Omega^i \varphi|, \\
 R_4 &= \doubleint_{\mathcal{D}}|L\varphi||\nablaslash\Omega^i\varphi||\Lb\Omega^i \varphi|.
\end{split}
\end{equation*}
We bound those terms one by one. For $R_1$, in view of Lemma \ref{lemma L2 of Lb phi}, we have
\begin{align*}
 R_1 &\leq \int_{0}^{\ub}(\|L \varphi\|_{L^{\infty}}+\|\nablaslash\varphi\|_{L^{\infty}}) \|\Lb\Omega^i\varphi\|^2_{L^2(\Cb_{\ub'})} d \ub' \\
 &\lesssim \delta^{\frac{3}{2}}|u|^{-2}M^3.
\end{align*}
For $R_2$, $R_3$ and $R_4$, in view of Proposition \ref{proposition L infinity and L4 estimates}, we bound the three factor in the integrands in $L^4$, $L^4$ and $L^2$, hence,
\begin{equation*}
R_2+R_3+ R_4 \lesssim \delta^{\frac{5}{4}}|u|^{-\frac{3}{2}}M^3.
\end{equation*}
We remark that for $R_4$, it is necessary to bound $\nablaslash\Omega^i\varphi$ in $L^4$ instead of $L^2$. In this way, one can gain an extra $\delta^\frac{1}{4}$.

We now move to $S$, i.e. the third integral terms on the right hand side of all equations \eqref{ES-Eb1-Eb2-Eb3-a}. A general form for the integrand can be written schematically as
\begin{equation*}
S = \doubleint_{\mathcal{D}}|\nabla \Omega^p \varphi||\nabla \Omega^q\varphi||\Lb\Omega^i \varphi|.
\end{equation*}
Notice that at least one $\nabla$ in this formula is not $L$. Thus, we can estimate this term by bounding the three factors in the integrands in $L^2$, $L^4$ and $L^4$. This leads to
\begin{equation*}
S \lesssim \delta^{\frac{5}{4}}|u|^{-\frac{3}{2}}M^3.
\end{equation*}

For $T$, i.e. the last integral terms on the right hand side of all equations \eqref{ES-Eb1-Eb2-Eb3-a},
\begin{equation*}
T = \doubleint_{\mathcal{D}} \frac{1}{r}|L\Omega^i \varphi||\Lb\Omega^i\varphi|,
\end{equation*}
we shall use a different approach. This is closely related the so called \textit{reductive structure} in Christodoulou's work \cite{Ch-08}. Roughly speaking, at this point, one have to proceed the estimates in a correct order and one has to rely the estimates derived in previous steps. This is not transparent in the current work, because instead of deriving the estimates in some kind of arbitrary order, the order of the estimates we do here is carefully chosen. We bound $T$ as follows
\begin{align*}
T &\lesssim \doubleint_{\mathcal{D}}\frac{1}{|u'|}|L\Omega^i\varphi||\Lb\Omega^i\varphi| \\
&\lesssim \doubleint_{\mathcal{D}}\frac{\delta}{|u'|^2}|L\Omega^i\varphi|^2+ \frac{1}{\delta}|\Lb\Omega^i\varphi|^2\\
&= \int_{u_0}^{u}\frac{\delta}{|u'|^2} \|L\Omega^i\varphi\|^2_{L^{2}(C_{u'})} du' + \frac{1}{\delta}\int_{0}^{\ub}\|\Lb\Omega^i\varphi\|^2_{L^2(\Cb_{\ub'})} d\ub'.
\end{align*}
The first term in last line has already been controlled in \eqref{ES-E1-E2-E3-b}, so we have
\begin{align*}
T &\lesssim \int_{u_0}^{u}\frac{\delta}{|u'|^2} (I_4^2 + \delta^{\frac{1}{4}}|u|^{-\frac{1}{2}}M^{3}) du' + \frac{1}{\delta}\int_{0}^{\ub}\|\Lb\Omega^i\varphi\|^2_{L^2(\Cb_{\ub'})} d\ub'\\
&= \delta|u|^{-1}I_4^2 + \delta^{\frac{5}{4}}|u|^{-\frac{3}{2}}M^3  + \frac{1}{\delta}\int_{0}^{\ub}\|\Lb\Omega^i\varphi\|^2_{L^2(\Cb_{\ub'})} d\ub'.
\end{align*}

In view of the size of initial data \eqref{L2 on initial surface 1}, we plug the above estimates into \eqref{ES-Eb1-Eb2-Eb3-a} to derive
\begin{equation*}
\int_{C_{u}}|\nablaslash\Omega^i\varphi|^{2}+\int_{\underline{C}_{\underline{u}}}|\Lb\Omega^i\varphi|^{2} \lesssim \delta|u|^{-1}I_4^2 + \delta^{\frac{5}{4}}|u|^{-\frac{3}{2}}M^3  + \frac{1}{\delta}\int_{0}^{\ub}\|\Lb\Omega^i\varphi\|^2_{L^2(\Cb_{\ub'})} d\ub'.
 \end{equation*}
Since $\|\Lb\Omega^i\varphi\|^2_{L^2(\Cb_{\ub})}$ also appears on the left hand side, a standard use of Gronwall's inequality removes the integral on the right hand side. This yields
\begin{equation}\label{ES-Eb1-Eb2-Eb3-b}
\|\nablaslash\Omega^i\varphi\|_{L^2(C_u)} + \|\Lb\Omega^i \varphi\|_{L^2(\Cb_{\ub})} \lesssim \delta^{\frac{1}{2}}|u|^{-\frac{1}{2}}I_4 + \delta^{\frac{5}{8}}|u|^{-\frac{3}{4}}M^{\frac{3}{2}}.
\end{equation}

Putting \eqref{ES-E1-E2-E3-b} and \eqref{ES-Eb1-Eb2-Eb3-b} together, we derive the energy estimates for one derivatives of $\varphi$ as follows,
\begin{equation}\label{ES E-Eb up to three derivative}
E_k(u,\ub) + \Eb_k(u, \ub) \lesssim I_4 +  \delta^{\frac{1}{8}}|u|^{-\frac{1}{4}} M^{\frac{3}{2}},
\end{equation}
for $k=1, 2$ and $3$.

\subsection{Energy Estimates for $F_{k}(u,\ub)$ and $\Fb_{k}(u,\ub)$ when $k =2, 3$ and $4$}\label{Section Estimates on Fb_3 and Fb_4}
We start with $F_{k}(u,\ub)$'s. By commuting $L$ and $\Omega$ with \eqref{Main Equation}, in view of \eqref{commute L Omega n with main equation},
we use the scheme in Section \ref{Energy estimates scheme} for this equation where we take $\phi = L\Omega^i\varphi$ with $i=0, 1, 2$ and $X = L$. We then have
\begin{equation}\label{ES-F3-F4}
\begin{split}
&\int_{C_{u}}|L^2 \Omega^i   \varphi|^{2}+\int_{\underline{C}_{\underline{u}}}|\nablaslash L \Omega^i\varphi|^{2} =\int_{C_{u_{0}}}|L^2 \Omega^i \varphi|^{2}  +\doubleint_{\mathcal{D}}Q(\nabla L \Omega^i\varphi, \nabla \varphi)L^2\Omega^i\varphi \\
& + \doubleint_{\mathcal{D}} \sum_{p+q\leq i,p< i, q<i} Q(\nabla L\Omega^{p}\phi,
\nabla\Omega^{q}\phi) L^2\Omega^i \varphi + \doubleint_{\mathcal{D}}\sum_{ p+q \leq i}
[L, Q](\nabla\Omega^{p}\phi, \nabla\Omega^{q}\phi)L^2 \Omega^i\varphi\\
& -\doubleint_{\mathcal{D}} \frac{1}{r^2}(\Lb \Omega^i\varphi -L \Omega^i\varphi)L^2 \Omega^i\varphi-\frac{2}{r}\laplacianslash \Omega^i\varphi L^2 \Omega^i\varphi-\frac{1}{r} \Lb L \Omega^i\varphi L^2\Omega^i\varphi.
\end{split}
\end{equation}
At this point, we rewrite the above equations as
\begin{equation*}
\begin{split}
\int_{C_{u}}|L^2 \Omega^i \varphi|^{2}+\int_{\underline{C}_{\underline{u}}}|\nablaslash L \Omega^i\varphi|^{2} =\int_{C_{u_{0}}}|L^2 \Omega^i \varphi|^{2}  +  R + S +T +U,
\end{split}
\end{equation*}
where $R$, $S$, $T$ and $U$ are defined in an obvious way.

First of all, we consider $R$, i.e. the second integral term on the right hand side of \eqref{ES-F3-F4}, in view of \eqref{null form bound}, it splits into
\begin{equation*}
\begin{split}
 R_1 &= \doubleint_{\mathcal{D}}(|\Lb \varphi|+|\nablaslash \varphi|)|L^2 \Omega^i \varphi|^2, \\
 R_2 &= \doubleint_{\mathcal{D}}(|\nablaslash \varphi|+|L \varphi|)|\Lb L \Omega^i \varphi||L^2 \Omega^i \varphi|,\\
 R_3 &= \doubleint_{\mathcal{D}}(|\nablaslash \varphi|+|\Lb \varphi| + |L \varphi|)|\nablaslash L \Omega^i \varphi||L^2 \Omega^i\varphi|.
\end{split}
\end{equation*}
We will bound the three factors in these integrands in $L^{\infty}$, $L^2$ and $L^2$ respectively.

For $R_1$, we simply bound $|L \varphi|$ and $|\nablaslash \varphi|$ in $L^\infty$ and obtain
\begin{equation*}
R_1  \lesssim \delta^{-\frac{7}{4}}|u|^{-\frac{1}{2}}M^3.
\end{equation*}

For $R_2$, we first need $L^2$ estimates on $L \Lb \Omega^i \varphi$. According to \eqref{Main Equation in null frame}, we have
\begin{align*}
 \|L\Lb\varphi\|_{L^2(C_u)} &\lesssim \|\laplacianslash \varphi\|_{L^2(C_u)}+\|\frac{1}{r} L\varphi\|_{L^2(C_u)}+\|\frac{1}{r}\Lb\varphi\|_{L^2(C_u)}\\
&\quad  +\|\Lb\varphi  L\varphi\|_{L^2(C_u)}+\|\Lb \varphi \nablaslash \varphi\|_{L^2(C_u)}+\|L\varphi \nablaslash \varphi\|_{L^2(C_u)}+\||\nablaslash \varphi|^2\|_{L^2(C_u)}
\end{align*}
For those quadratic terms, we bound one of them in $L^\infty$, thus, we have
\begin{equation}\label{estimates on L Lb phi}
 \|L\Lb\varphi\|_{L^2(C_u)} \lesssim |u|^{-1}M.
\end{equation}
For $L\Lb \Omega^i \varphi$'s with $i=1,2$, we can proceed exactly in the same way to derive (we also use Sobolev inequalities)
\begin{equation}\label{estimates on L Lb Omega phi}
\begin{split}
 \|L\Lb \Omega^i \varphi\|_{L^2(C_u)} &\lesssim |u|^{-1}M,\\
  \|L\Lb\Omega^{i-1} \varphi\|_{L^4(S_{\ub,u})} &\lesssim \delta^{-\frac{1}{2} } |u|^{-\frac{3}{2}} M.
  \end{split}
\end{equation}
Therefore, we bound $|\nablaslash \varphi|$ and $|L \varphi|$ in $L^\infty$, $|\Lb L \Omega^i \varphi|$ and $|L^2 \Omega^i \varphi|$ in $L^2$, this yields
\begin{align*}
R_2 \lesssim \delta^{-\frac{3}{2}}|u|^{-1}M^3.
\end{align*}
For $R_3$, similarly, we have
\begin{align*}
R_3 \lesssim \delta^{-\frac{3}{2}}|u|^{-1}M^3.
\end{align*}

Secondly, we consider $S$, i.e. the third integral term on the right hand side of \eqref{ES-F3-F4}. It is bounded by the sum of the following terms
\begin{equation*}
\begin{split}
S_1 &= \doubleint_{\mathcal{D}}(|\Lb \Omega^q\varphi|+|\nablaslash \Omega^q \varphi|)|L^2\Omega^p \varphi||L^2 \Omega^i \varphi|, \\
S_2 &= \doubleint_{\mathcal{D}}(|\nablaslash \Omega^q\varphi|+|L \Omega^q\varphi|)|\Lb L \Omega^p \varphi||L^2 \Omega^i \varphi|,\\
S_3 &= \doubleint_{\mathcal{D}}(|\nablaslash\Omega^q \varphi|+|\Lb \Omega^q\varphi| + |L \Omega^q\varphi|)|\nablaslash L \Omega^p  \varphi||L^2 \Omega^j\varphi|,
\end{split}
\end{equation*}
where $p+q\leq i$, $p<i$ and $q < i$. Since the numbers of derivatives in the first factors are not saturated, we will bound the three factors in these integrands in $L^{4}$, $L^4$ and $L^2$ respectively.

For $S_1$, we have
\begin{align*}
 S_1  &\leq \int_0^{\ub} \int_{u_0}^{u}(\|\nablaslash\Omega^q\varphi\|_{L^4(S_{\ub', u'})}+\|\Lb\Omega^q\varphi\|_{L^4(S_{\ub',u'})})\|L^2\Omega^p\varphi\|_{L^4(S_{\ub',u'})} \|L^2\Omega^i\varphi\|_{L^2(S_{\ub', u'})} \\
&\leq \delta^{\frac{1}{4}} M \int_0^{\ub} \int_{u_0}^{u} (|u'|^{-\frac{3}{2}} \|L^2\Omega^j \varphi\|_{L^{2}(S_{\ub', u'})} +|u'|^{-\frac{1}{2}}\|\nablaslash L^2\Omega^j \varphi \|_{L^{2}(S_{\ub', u'})}) \|L^2\Omega^i\varphi\|_{L^2(S_{\ub', u'})},
\end{align*}
where $j = q$, $q+1$ and we have used the following Sobolev inequalities for last line,
 \begin{equation*}
\|L^2 \Omega^j\varphi \|_{L^{4}(S_{\ub,u})} \lesssim  |u|^{-\frac{1}{2}}\|L^2 \Omega^j\varphi\|_{L^{2}(S_{\ub, u})} +|u|^{\frac{1}{2}}\|\nablaslash L^2\Omega^j \varphi \|_{L^{2}(S_{\ub, u})}.
 \end{equation*}
Thus, we have
\begin{equation*}
\begin{split}
 S_1  &\lesssim \delta^{\frac{1}{4}} M \int_{u_0}^{u} |u'|^{-\frac{3}{2}}\|L^2\Omega^j \varphi\|_{L^{2}(C_{u'})}  \|L^2\Omega^i\varphi\|_{L^2(C_{u'})}+|u'|^{-\frac{3}{2}}\|L^2\Omega^{j+1}\varphi\|_{L^2(C_{u'})}\|L^2\Omega^i\varphi\|_{L^2(C_{u'})} \\
 &\lesssim \delta^{-\frac{7}{4}}|u|^{-\frac{1}{2}}M^3.
 \end{split}
\end{equation*}
For $S_2$ and $S_3$, similarly, we obtain
\begin{equation*}
S_2+S_3 \lesssim \delta^{-\frac{3}{2}} |u|^{-1} M^3.
\end{equation*}

Thirdly, we consider $T$, i.e. the fourth integral terms on the right hand side of \eqref{ES-F3-F4}. It is bounded by
\begin{equation*}
T = \doubleint_{\mathcal{D}}\frac{1}{|u|} \sum_{p+q \leq i}(|\nablaslash \Omega^p\varphi||L \Omega^q\varphi|+|\nablaslash \Omega^p\varphi||\Lb \Omega^q\varphi| + |\Lb\Omega^p\varphi||L \Omega^q \varphi| ) |L^2 \Omega^i \varphi|,
\end{equation*}
The strategy is to control the three factors in the integrands either in $L^\infty$, $L^2$ and $L^2$ or in $L^4$, $L^4$ and $L^2$ respectively, depending wether the number of derivatives are saturated or not. We omit the details since the proof is exactly the same as for $R$ and $S$ terms. We obtain
\begin{equation*}
T \lesssim \delta^{-\frac{3}{2}}|u|^{-1}M^3.
\end{equation*}

Finally, we consider $U$, i.e. the last integral term in \eqref{ES-F3-F4}, we simply estimate two factors in the integrands in $L^2$ and $L^2$. This gives
\begin{equation*}
U \lesssim \delta^{-1}|u|^{-1}M^2+\delta^{-\frac{1}{2}}|u|^{-\frac{3}{2}}M^2.
\end{equation*}

Putting all the estimates back to \eqref{ES-F3-F4}, in view of the size of data \eqref{L2 on initial surface 2}, we have
\begin{equation*}
\int_{C_{u}}|L^2 \Omega^i \varphi|^{2}+\int_{\underline{C}_{\underline{u}}}|\nablaslash L \Omega^i \varphi|^{2} \lesssim \delta^{-2}I_4^2  + \delta^{-\frac{7}{4}}|u|^{-\frac{1}{2}}M^3.
\end{equation*}
for $i=0,1$ and $2$. This is equivalent to say that
\begin{equation}\label{ES-ES-F3-F4-b}
F_{k}(u) \lesssim I_4  + \delta^{\frac{1}{8}}|u|^{-\frac{1}{4}}M^\frac{3}{2},
\end{equation}
for $k=2,3$ and $4$.\\

We now derive the estimates for $\Fb_{k}(u,\ub)$'s. Since the proof can almost be converted word by word from the proof we just preformed for $F_{k}(u,\ub)$'s, in stead of giving all the details, we only sketch the idea.

By commuting $\Lb, \Omega$ with \eqref{Main Equation}, in view of \eqref{commute Lb Omega n with main equation}, one can apply the scheme in Section \ref{Energy estimates scheme} for this equation by take $\phi = \Lb\Omega^i\varphi, i=0, 1, 2,$ and $X = \Lb$. Thus, \eqref{fundamental energy identity} reads as
\begin{equation}\label{ES-Fb3-Fb4-a}
\begin{split}
&\int_{C_{u}}|\nablaslash \Lb \Omega^i \varphi|^{2}+\int_{\underline{C}_{\underline{u}}}|\Lb^2 \Omega^i\varphi|^{2} =\int_{C_{u_{0}}}|\nablaslash \Lb \Omega^i \varphi|^{2}  +\doubleint_{\mathcal{D}}Q(\nabla \Lb \Omega^i\varphi, \nabla \varphi)\Lb^2\Omega^i\varphi  \\
&+ \doubleint_{\mathcal{D}} \sum_{p+q\leq i,p<i,q<i} Q(\nabla \Lb\Omega^{p}\phi,
\nabla\Omega^{q}\phi) \Lb^2\Omega^i \varphi + \doubleint_{\mathcal{D}}\sum_{ p+q \leq i}
[\Lb, Q](\nabla\Omega^{p}\phi, \nabla\Omega^{q}\phi)\Lb^2 \Omega^i\varphi\\
& +\doubleint_{\mathcal{D}}\frac{1}{r^2}(\Lb \Omega^i\varphi -L \Omega^i\varphi)\Lb^2 \Omega^i\varphi-\frac{2}{r}\laplacianslash \Omega^i\varphi \Lb^2 \Omega^i\varphi-\frac{1}{r}  L \Lb \Omega^i\varphi \Lb^2\Omega^i\varphi .
\end{split}
\end{equation}
At this point, we rewrite the above equations as
\begin{equation*}
\begin{split}
\int_{C_{u}}|\nablaslash \Lb \Omega^i \varphi|^{2}+\int_{\underline{C}_{\underline{u}}}|\Lb^2 \Omega^i\varphi|^{2} =\int_{C_{u_{0}}}|\nablaslash \Lb \Omega^i \varphi|^{2} +  R + S +T +U,
\end{split}
\end{equation*}
where $R$, $S$, $T$ and $U$ are defined in an obvious way.

For $R$, we bound it by the sum of the following terms
\begin{equation*}
\begin{split}
 R_1 &= \doubleint_{\mathcal{D}}(|L \varphi|+|\nablaslash \varphi|)|\Lb^2 \Omega^i \varphi|^2, \\
 R_2 &= \doubleint_{\mathcal{D}}(|\nablaslash \varphi|+|\Lb \varphi|)|L \Lb \Omega^i \varphi||\Lb^2 \Omega^i \varphi|,\\
 R_3 &= \doubleint_{\mathcal{D}}(|\nablaslash \varphi|+|\Lb \varphi| + |L \varphi|)|\nablaslash \Lb \Omega^i \varphi||\Lb^2 \Omega^i\varphi|.
\end{split}
\end{equation*}
We then bound the three factors in the above integrands in $L^{\infty}$, $L^2$ and $L^2$ respectively. This will yield directly
\begin{equation*}
R_1  \lesssim \delta^{\frac{1}{2}}|u|^{-2}M^3,
\end{equation*}
\begin{equation*}
R_2 \lesssim \delta^{\frac{3}{4}}|u|^{-\frac{5}{2}} M^3,
\end{equation*}
and
\begin{equation*}
R_3 \lesssim \delta |u|^{-3}M^3.
\end{equation*}

For $S$, it is bounded by the sum of the following terms
\begin{equation*}
\begin{split}
S_1 &= \doubleint_{\mathcal{D}}(|L \Omega^q \varphi|+|\nablaslash \Omega^q \varphi|)|\Lb^2 \Omega^p \varphi||\Lb^2 \Omega^i \varphi|, \\
S_2 &= \doubleint_{\mathcal{D}}(|\nablaslash \Omega^q\varphi|+|\Lb \Omega^q\varphi|)| L \Lb \Omega^p\varphi||\Lb^2 \Omega^i \varphi|,\\
S_3 &= \doubleint_{\mathcal{D}}(|\nablaslash\Omega^q \varphi|+|\Lb \Omega^q\varphi| + |L \Omega^q\varphi|)|\nablaslash \Lb  \Omega^p \varphi||\Lb^2 \Omega^i\varphi|,
\end{split}
\end{equation*}
where $p+q\leq i$, $p<i$ and $q < i$.

Since the numbers of derivatives in the first factors are not saturated, we can bound the three factors in these integrands in $L^{4}$, $L^4$ and $L^2$ respectively. This yields
\begin{equation*}
S_1  \lesssim \delta^{\frac{1}{2}}|u|^{-2}M^3,
\end{equation*}
and
\begin{equation*}
S_2+S_3 \lesssim \delta^{\frac{3}{4}}|u|^{-\frac{5}{2}} M^3.
\end{equation*}

For $T$, typically, there are bounded by
\begin{equation*}
 S_{7}= \doubleint_{\mathcal{D}}\frac{1}{|u|} \sum_{p+q \leq i}(|\nablaslash \Omega^p\varphi||L \Omega^q\varphi|+|\nablaslash \Omega^p\varphi||\Lb \Omega^q\varphi| + |\Lb\Omega^p\varphi||L \Omega^q \varphi| ) |L^2 \Omega^i \varphi|,
\end{equation*}
We control the three factors in the integrands either in $L^\infty$, $L^2$ and $L^2$ or in $L^4$, $L^4$ and $L^2$ respectively and obtain
\begin{equation*}
T \lesssim \delta^{\frac{3}{4}}|u|^{-\frac{5}{2}} M^3.
\end{equation*}

For $U$, we control the two factors in their integrands in $L^2$ and $L^2$ and obtain
\begin{equation*}
U \lesssim \delta^{\frac{1}{2} }|u|^{-2}M^2+\delta|u|^{-\frac{5}{2}}M^2.
\end{equation*}

Putting all those estimates together, in view of \eqref{L infinity on Lb nablaslash nablaslash phi on C_0} and the fact that $|u|\leq 1$, we derive
\begin{equation*}
\int_{\underline{C}_{\underline{u}}}|\Lb^2 \Omega^i\varphi|^{2} \lesssim \delta^{2}|u_0|^{-4}I_4^2  + \delta^{\frac{1}{2} }|u|^{-2}M^3.
\end{equation*}
for $i=0,1,2$. This is equivalent to say that
\begin{equation}\label{ES-Fb3-Fb4-b}
\Fb_{k}(\ub) \lesssim \delta|u_0|^{\frac{-3}{2}}I_4 + \delta^{\frac{1}{4}}|u|^{-\frac{1}{2}} M^\frac{3}{2},
\end{equation}
for $k=2,3$ and $4$.

\subsection{Estimates for $E_4(u,\ub)$ and $\Eb_4(u,\ub)$}
First of all, we commute $\Omega$ three times with \eqref{Main Equation}, that is, taking $n=3$ in \eqref{commute n Omega with main equation}, this yields
\begin{equation*}
\Box \Omega^{3} \varphi = \sum_{p+q \leq 3}Q(\nabla \Omega^p \varphi, \nabla \Omega^q \varphi).
\end{equation*}

For $E_4(u,\ub)$, we use the scheme in Section \ref{Energy estimates scheme} for this equation by taking $\phi = \Omega^3 \varphi$ and $X = L$. In view of \eqref{fundamental energy identity}, we have
\begin{equation}\label{ES-4-1-a}
\begin{split}
\int_{C_{u}}|L \Omega^3 \varphi|^{2}+\int_{\underline{C}_{\underline{u}}}&|\nablaslash \Omega^3 \varphi|^{2} =\int_{C_{u_{0}}}|L\Omega^3 \varphi|^{2} + \doubleint_{\mathcal{D}}L \Omega^3 \varphi Q(\nabla \Omega^3 \varphi, \nabla \varphi) \\
& + \doubleint_{\mathcal{D}}L \Omega^3 \varphi \sum_{p+q \leq 3, p<3, q<3}Q(\nabla \Omega^p \varphi, \nabla \Omega^q \varphi)+\doubleint_{\mathcal{D}} \frac{1}{r} \Lb \Omega^3 \varphi \cdot L \Omega^3\varphi.
\end{split}
\end{equation}
At this point, we rewrite the above equations as
\begin{equation*}
\begin{split}
\int_{C_{u}}|L \Omega^3 \varphi|^{2}&+\int_{\underline{C}_{\underline{u}}}|\nablaslash \Omega^3 \varphi|^{2} =\int_{C_{u_{0}}}|L\Omega^3 \varphi|^{2}+ R + S +T.
\end{split}
\end{equation*}
where $R$, $S$ and $T$ are defined in an obvious way.

We claim that the estimates for $S$ and $T$ are easy and we can bound them by
\begin{equation*}
S \lesssim \delta^{\frac{1}{4}}|u|^{-\frac{1}{2}}M^3,
\end{equation*}
and
\begin{equation*}
T \lesssim \delta|u|^{-2}M^2.
\end{equation*}
To be more precise: since the numbers of derivatives for the integrands of $S$ are not saturated (i.e. only one term has four derivatives), we can bound the three factors in the integrands in $L^{4}$, $L^4$ and $L^2$ respectively; for $T$, we simply bound the integrands in two $L^2$'s by H\"{o}lder's inequality. The actual proof goes in the same way as in previous the section and we omit the details.

It remains to bound $R$. This terms can be bounded by the sum of the following terms before,
\begin{equation*}
\begin{split}
R_1 &= \doubleint_{\mathcal{D}}(|\Lb \varphi|+|\nablaslash \varphi|)|L \Omega^3 \varphi|^2, \\
R_2 &= \doubleint_{\mathcal{D}}(|L \varphi|+|\nablaslash \varphi|)|\Lb \Omega^3 \varphi||L \Omega^3 \varphi|, \\
R_3 &= \doubleint_{\mathcal{D}}(|\nablaslash \varphi|+|\Lb \varphi| + |L \varphi|)|\nablaslash \Omega^3 \varphi||L \Omega^3\varphi|.
\end{split}
\end{equation*}

We remark that, because this is about the top order derivative estimates, we have to place the last two terms in the above integrands in $L^2$ and the first one in $L^\infty$. In such a way, we can easily derive
\begin{equation*}
R_1  \lesssim \delta^{\frac{1}{4}}|u|^{-\frac{1}{2}}M^3,
\end{equation*}
and
\begin{equation*}
R_2 \lesssim \delta^{\frac{1}{2}}|u|^{-1}M^3.
\end{equation*}

The estimates for $R_3$ are more difficult and require the knowledge of all the previous estimates derived so far. Before going into the details, we would like to explain how the difficulties appear. It is intimately related to the relaxation of the propagation estimates. One may expect $\|\nablaslash \Omega^3 \varphi\|_{L^2(C_u)}$ behaves as $\delta^{\frac{1}{2}}$ in view of the initial data. But in reality, because we are using a relaxed version of propagation estimates, we lose automatically $\delta^{\frac{1}{2}}$. Therefore, if we treat $R_3$ in the same way as for $R_1$ and $R_2$, we will not get any positive power in $\delta$ and therefore we can not close the bootstrap argument.

To around the difficulties, we recall that in previous sections those $L^\infty$ estimates (say on $L\varphi$) are directly derived from the bootstrap assumptions via the Sobolev inequalities. The key observation is, if we make use of the estimates derived so far in previous sections instead of the bootstrap assumptions, we can indeed improve the $L^\infty$ estimates for $L\varphi$. This improvement will be just good enough to enable us to close the argument.

We first improve the $L^{4}$ estimates for $L\Omega\varphi$, according to \eqref{ES E-Eb up to three derivative} and Sobolev inequalities, we have
\begin{equation}\label{Improved-L4-L-2-2}
\begin{split}
 |u|^{\frac{1}{2}}\| L\Omega\varphi \|_{L^{4}(S_{\ub,u})} &\lesssim  \| L^{2}\Omega\varphi\|^{\frac{1}{2}}_{L^{2}(C_{u})}(\| L\Omega\varphi \|^{\frac{1}{2}}_{L^{2}(C_{u})}
 +|u|^{\frac{1}{2}}\| \nablaslash L\Omega\varphi \|^{\frac{1}{2}}_{L^{2}(C_{u})}) \\
&\lesssim \delta^{-\frac{1}{2}}(I_{4}^{\frac{1}{2}}+\delta^{\frac{1}{16}}
|u|^{-\frac{1}{8}}M^{\frac{3}{4}})(I_4^{\frac{1}{2}} + \delta^{\frac{1}{16}}|u|^{-\frac{1}{8}}M^{\frac{3}{4}})\\
&\lesssim \delta^{-\frac{1}{2}}(I_4+\delta^{\frac{1}{8}}
|u|^{-\frac{1}{4}}M^{\frac{3}{2}}).
\end{split}
\end{equation}
This implies better $L^{\infty}$ estimates for $L\varphi$, once again via Sobolev inequalities, as follows
\begin{equation}\label{Improved-L-Infinity-L-2}
\begin{split}
 \|L\varphi\|_{L^\infty} &\lesssim  |u|^{-\frac{1}{2}}\| L\varphi \|_{L^{4}(S_{\ub,u})}
 +|u|^{\frac{1}{2}}\| L\nablaslash\varphi \|_{L^{4}(S_{\ub,u})}\\
 & \lesssim \delta^{-\frac{1}{2}}|u|^{-1} (I_4+\delta^{\frac{1}{8}}
|u|^{-\frac{1}{4}}M^{\frac{3}{2}}).
\end{split}
\end{equation}
We now proceed to bound $R_3$ and we only consider the main terms  with $L \varphi$ (the others are much easier to control):
\begin{align*}
R_3 &\leq \doubleint_{\mathcal{D}}|L \varphi|^2|u'|^2|\nablaslash\Omega^3\varphi|^2
+|u'|^{-2}|L\Omega^3\varphi|^2du'd\ub'\\
&\leq \int_{0}^{\delta}\delta^{-1}(I_4+\delta^{\frac{1}{8}}
|u|^{-\frac{1}{4}}M^{\frac{3}{2}})^2\|\nablaslash\Omega^3\varphi\|^2_{L^2(\Cb_{\ub'})}d\ub'
+\int_{u_0}^{u}|u'|^{-2}\|L\Omega^3\varphi\|^2_{L^2(C_{u'})} du'.
\end{align*}

Finally, we put all the estimates together, in view of the size of the initial data on $C_{u_0}$, we have
\begin{equation*}
\begin{split}
\|L\Omega^3\varphi\|^2_{L^2(C_{u})}+\|\nablaslash\Omega^3\varphi\|^2_{L^2(\Cb_{\ub})} &\lesssim I_{4}^{2}+ \delta^{\frac{1}{4}}|u|^{-\frac{1}{2}}M^3+\int_{u_0}^{u}|u'|^{-2}\|L\Omega^3\varphi\|^2_{L^2(C_{u'})} du'\\
&+\int_{0}^{\delta}\delta^{-1}(I_4+\delta^{\frac{1}{8}}
|u|^{-\frac{1}{4}}M^{\frac{3}{2}})^2\|\nablaslash\Omega^3\varphi\|^2_{L^2(\Cb_{\ub'})}d\ub'.
\end{split}
\end{equation*}
Thus, thanks to Gronwall's inequality,
\begin{equation*}
\begin{split}
&\quad \quad |u|^3\|L\nablaslash^3\varphi\|_{L^2(C_u)} + \||u^3|\nablaslash^4 \varphi\|_{L^2(\Cb_{\ub})}\\
&\lesssim (\exp|u|^{-1}+\exp (I_4+\delta^{\frac{1}{8}}
|u|^{-\frac{1}{4}}M^{\frac{3}{2}})^2 )(I_4 +\delta^{\frac{1}{8}}|u|^{-\frac{1}{4}}M^{\frac{3}{2}}).
\end{split}
\end{equation*}
Therefore, if $\delta$ is sufficiently small, we have
\begin{equation*}
\begin{split}
|u|^3\|L\nablaslash^3\varphi\|_{L^2(C_u)} + \||u^3|\nablaslash^4 \varphi\|_{L^2(\Cb_{\ub})} \lesssim I_4 +\delta^{\frac{1}{8}}|u|^{-\frac{1}{4}}M^{\frac{3}{2}}.
\end{split}
\end{equation*}
Equivalently,
\begin{equation}\label{ES-4-1-b}
E_4(u,\ub)\lesssim I_4 +  \delta^{\frac{1}{8}}|u|^{-\frac{1}{4}} M^{\frac{3}{2}}.
\end{equation}
This is the desired estimates for $E_4(u,\ub)$.\\

For $\Eb_4(u,\ub)$, we switch $X$ to $\Lb$. In view of \eqref{fundamental energy identity}, we have

\begin{equation}\label{ES-4-2-a}
\begin{split}
\int_{C_{u}}|\nablaslash \Omega^3 \varphi|^{2}+\int_{\underline{C}_{\underline{u}}}&|\Lb \Omega^3 \varphi|^{2} =\int_{C_{u_{0}}}|\nablaslash \Omega^3 \varphi|^{2}  + \doubleint_{\mathcal{D}}\Lb \Omega^3 \varphi Q(\nabla \Omega^3 \varphi, \nabla \varphi) \\
& + \doubleint_{\mathcal{D}}\Lb \Omega^3 \varphi \sum_{p+q \leq 3, p<3, q<3}Q(\nabla \Omega^p \varphi, \nabla \Omega^q \varphi)-\doubleint_{\mathcal{D}} \frac{1}{r} \Lb \Omega^3 \varphi \cdot L \Omega^3\varphi.
\end{split}
\end{equation}
At this point, we rewrite the above equations as
\begin{equation*}
\begin{split}
\int_{C_{u}}|\nablaslash \Omega^3 \varphi|^{2}+\int_{\underline{C}_{\underline{u}}}|\Lb \Omega^3 \varphi|^{2} =\int_{C_{u_{0}}}|\nablaslash \Omega^3 \varphi|^{2}+ R + S +T.
\end{split}
\end{equation*}
where $R$, $S$ and $T$ are defined in an obvious way.

We claim that the estimates for $S$ are easy since the numbers of derivatives for the integrands of $S$ are not saturated. We bound the three factors in the integrands in $L^{4}$, $L^4$ and $L^2$ respectively and we derive
\begin{equation*}
S \lesssim \delta^{\frac{5}{4}}|u|^{-\frac{3}{2}}M^3.
\end{equation*}

We can bound $R$ more or less as before. First of all, it is bounded by the sum of the following terms

\begin{equation*}
\begin{split}
R_1 &= \doubleint_{\mathcal{D}}(|L \varphi|+|\nablaslash \varphi|)|\Lb \Omega^3 \varphi|^2, \\
R_2 &= \doubleint_{\mathcal{D}}(|\Lb \varphi|+|\nablaslash \varphi|)|L \Omega^3 \varphi||\Lb \Omega^3 \varphi|, \\
R_3 &= \doubleint_{\mathcal{D}}(|\nablaslash \varphi|+|\Lb \varphi| + |L \varphi|)|\nablaslash \Omega^3 \varphi||\Lb \Omega^3\varphi|.
\end{split}
\end{equation*}
Once again, except for the last term in $R_3$, all the other terms are easy to control so we ignore all the them and assume
\begin{equation*}
R_3 = \doubleint_{\mathcal{D}} |L \varphi| |\nablaslash \Omega^3 \varphi||\Lb \Omega^3\varphi|.
\end{equation*}
Therefore, we repeat the previous argument and make use the improved $L^\infty$ estimates for $L \varphi$, we obtain
\begin{align*}
R_3 &\leq \doubleint_{\mathcal{D}}|L \varphi|^2|u'|^2|\Lb\Omega^3\varphi|^2
+|u'|^{-2}|\nablaslash\Omega^3\varphi|^2du'd\ub'\\
&\leq \int_{0}^{\delta}\delta^{-1}(I_4+\delta^{\frac{1}{8}}
|u|^{-\frac{1}{4}}M^{\frac{3}{2}})^2\|\Lb\Omega^3\varphi\|^2_{L^2(\Cb_{\ub'})}d\ub'
+\int_{u_0}^{u}|u'|^{-2}\|\nablaslash\Omega^3\varphi\|^2_{L^2(C_{u'})} du'.
\end{align*}

For $T$, thanks to \eqref{ES-4-1-b}, we have
\begin{equation*}
T \lesssim \delta|u|^{-1}I_4^2+ \delta^{\frac{5}{4}}|u|^{-\frac{3}{2}}M^3+ \frac{1}{\delta}\int_{0}^{u}\|\Lb\Omega^{3}\varphi\|^2_{L^2(\Cb_{\ub'})} d\ub'.
\end{equation*}
Putting all the estimates together, we have
\begin{equation*}
\begin{split}
\|\nablaslash\Omega^3\varphi\|^2_{L^2(C_{u})}+\|\Lb\Omega^3\varphi\|^2_{L^2(\Cb_{\ub})} &\lesssim \delta|u|^{-1}I_4^2+ \delta^{\frac{5}{4}}|u|^{-\frac{3}{2}}M^3+\int_{u_0}^{u}|u'|^{-2}\|\nablaslash\Omega^3\varphi\|^2_{L^2(C_{u'})} du'\\
&+\int_{0}^{\delta}\delta^{-1}(1+(I_4+\delta^{\frac{1}{8}}
|u|^{-\frac{1}{4}}M^{\frac{3}{2}})^2)\|\Lb\Omega^3\varphi\|^2_{L^2(\Cb_{\ub'})}d\ub'
\end{split}
\end{equation*}
Thus, thanks to Gronwall's inequality, we obtain
\begin{equation}\label{ES-4-2-b}
\begin{split}
\|\nablaslash\Omega^3\varphi\|_{L^2(C_u)} + \|\Lb \Omega^3 \varphi\|_{L^2(\Cb_{\ub})}
 &\lesssim\delta^{\frac{1}{2}}|u|^{-\frac{1}{2}}I_4+ \delta^{\frac{5}{8}}|u|^{-\frac{3}{4}}M^{\frac{3}{2}}.
\end{split}
\end{equation}

We then combine \eqref{ES-4-1-b} and \eqref{ES-4-2-b} to conclude
\begin{equation}\label{ES fourth derivative E4 Eb4}
E_4(u) + \Eb_4(\ub) \lesssim I_4 +  \delta^{\frac{1}{8}}|u|^{-\frac{1}{4}} M^{\frac{3}{2}}.
\end{equation}

\subsection{End of the Bootstrap Argument}\label{Section End of the Bootstrap Argument}
We add the estimates in previous sections together, since $|u| \geq 1$, we derive
\begin{equation*}
\sum_{i=1}^4 [E_i(u) + \Eb_i(\ub)]+\sum_{j=2}^4 [F_j(u) + \Fb_j(\ub)] \lesssim I_4 + \delta^{\frac{1}{8}}M^{\frac{3}{2}}.
\end{equation*}
By the definition of $M$ from the bootstrap assumption \eqref{bootstrap assumption}, we obtain
\begin{equation*}
M \lesssim I_4 + \delta^{\frac{1}{8}}M^{\frac{3}{2}}.
\end{equation*}
By choosing $\delta$ suitably small depending on the quantities $I_4$, we conclude that, there is a constant $C(I_4)$ depending only on $I_4$, such that
\begin{equation*}
\sum_{i=1}^4 [E_i(u) + \Eb_i(\ub)]+\sum_{j=2}^4 [F_j(u) + \Fb_j(\ub)] \leq C(I_4).
\end{equation*}
Therefore, we have completed the proof of \textbf{Main A priori Estimates}.

\subsection{Higher Order Derivative Estimates}\label{Section Higher Order Derivative Estimates}
For higher order derivative estimates, the argument is completely analogous, in fact much simpler, because we have already closed the bootstrap argument and we can simply use an induction argument to derive estimates for each order. Therefore, we shall omit the detail and only sketch the proof. We introduce a family of energy flux norms for higher order derivatives:
\begin{equation*}
\begin{split}
E_{k}(u,\ub) &=|u|^{k-1} \|L \nablaslash^{k-1} \varphi\|_{L^2(C_u)} + \delta^{-\frac{1}{2}} |u|^{k-\frac{1}{2}}\|\nablaslash^{k} \varphi\|_{L^2(C_u)},\\
\Eb_{k}(u,\ub) &=|u|^{k-1} \|\nablaslash^{k} \varphi\|_{L^2(\Cb_{\ub})} + \delta^{-\frac{1}{2}} |u|^{\frac{1}{2}}\||u|^{k-1}\Lb \nablaslash^{k-1} \varphi\|_{L^2(\Cb_{\ub})},
\end{split}
\end{equation*}
for all $k \geq 1$. Similar to the lower order derivatives cases, we also need a family of flux norms involving at least two null derivatives:
\begin{equation*}
\begin{split}
F_{k}(u,\ub) &=\delta|u|^{k-2} \|L^2 \nablaslash^{k-2}\varphi\|_{L^2(C_u)}, \\
\Fb_{k}(u,\ub) &=|u|^{\frac{1}{2}} \||u|^{k-2}\Lb^2 \nablaslash^{k-2}\varphi\|_{L^2(\Cb_{\ub})},
\end{split}
\end{equation*}
for all $k \geq 2$.

To achieve the high order derivative estimates, we will perform an induction argument. The beginning cases for the induction argument have already been verified, this is indeed the \textbf{Main A priori Estimates} obtained earlier, that is,
\begin{equation*}
\sum_{i=1}^4 [E_i(u,\ub) + \Eb_i(u,\ub)]+\sum_{j=2}^4 [F_j(u,\ub) + \Fb_j(u,\ub)] \leq C(I_{4}),
\end{equation*}
where $I_4$ is the size of the data up to four derivatives. The higher order estimates are formulated as follows:
\begin{proposition}\label{main estimate high estimates by induction} If $\delta$ is sufficiently small which may depend only on $k$, for all data of \eqref{Main Equation} and all $I_{n+2} \in \mathbb{R}_{>0}$ satisfying
\begin{equation}\label{initial bound of high order}
\sum_{i=1}^{n+2} E_i(u_{0},\delta)  + \sum_{j=2}^{n+2} F_j(u_{0},\delta)  \leq I_{n+2},
\end{equation}
there exists a constant $C(I_{n+2})$ depending only on $I_{n+2}$ (in particular, not on $\delta$ and $u_0$), so that
\begin{equation}\label{main estimates of high order}
 [E_{n+2}(u,\ub) + \Eb_{n+2}(u,\ub)]+ [F_{n+2}(u,\ub) + \Fb_{n+2}(u,\ub)] \leq C(I_{n+2}),
\end{equation}
for all $u \in [u_0, -1]$ and $\ub \in [0, \delta]$ in the sense of a priori estimates.
\end{proposition}
\begin{remark}
The index in $I_{n+2}$ indicates the number of derivatives needed in the energy. The small parameter $\delta$ may depend on $n$. In applications, since we only need the bound on at least $10$ derivatives on the solutions, we can ignore this dependence.
\end{remark}

\begin{proof} We now sketch the proof. Once again, we make the following bootstrap assumption:
\begin{equation}\label{nth bootstrap assumption}
[E_{n+2}(u,\ub) + \Eb_{n+2}(u,\ub)]+ [F_{n+2}(u,\ub) + \Fb_{n+2}(u,\ub)] \leq M,
\end{equation}
for all $u$ and $\ub$ where $M$ is sufficiently large.

We proceed as before. First of all, we can derive preliminary estimates for higher order derivatives of $\varphi$, i.e. the $L^\infty$ estimates up to $n$-th order derivatives and $L^4$ estimates up to $(n+1)$-th order derivatives. Those estimates are simply from the Sobolev inequalities and they are listed as follows
\begin{equation*}
 \delta^{\frac{1}{2}}|u|^{i}\|L\nablaslash^{i-1}\varphi\|_{L^\infty} + \delta^{-\frac{1}{4}}|u|^{\frac{3}{4}+i}\|\nablaslash^{i}\varphi\|_{L^\infty} +  \delta^{-\frac{1}{4}}|u|^{\frac{3}{2}}\||u|^{i-1}\Lb\nablaslash^{i-1}\varphi\|_{L^\infty} \lesssim M,
\end{equation*}
\begin{equation*}
\delta^{\frac{1}{2}}|u|^{\frac{1}{2}+j}\| L\nablaslash^{j} \varphi \|_{L^{4}(S_{\ub,u})} +\delta^{-\frac{1}{4}}|u|^{\frac{5}{4}+j} \|\nablaslash^{j+1} \varphi\|_{L^{4}(S_{\ub,u})} + \delta^{-\frac{1}{4}}|u|^{2}\||u|^{j}\Lb\nablaslash^{j}\varphi \|_{L^{4}(S_{\ub,u})} \lesssim M,
\end{equation*}
for all $i, j \in \{2,3, \cdots, n\}$. In fact, based on the induction argument, we know that if $i$ or $j$ strictly less than $n$, we can replace the right hand sides of the above estimates by a constant depending only on $I_{n+1}$ instead of $M$.

Secondly, we can perform the similar arguments as in previous sections to obtain energy estimates, this lead to the following estimates
\begin{equation*}
\begin{split}
F_{n+2}(u,\ub) & \lesssim I_{n+2}  + \delta^{\frac{1}{8}}|u|^{-\frac{1}{4}}M^\frac{3}{2},\\
\Fb_{n+2}(u,\ub) &\lesssim  I_{n+2}\delta^{\frac{1}{2}}+\delta^{\frac{1}{8}}|u|^{-\frac{1}{4}}M^{2},\\
E_{n+2}(u,\ub) + \Eb_{n+2}(u,\ub) &\lesssim I_{n+2} +  \delta^{\frac{1}{8}}|u|^{-\frac{1}{4}} M^{\frac{3}{2}}.
\end{split}
\end{equation*}
Therefore, we can complete the proof by taking a sufficiently small $\delta$.
\end{proof}

\section{Existence of Solutions}\label{Section existence and uniqueness}

\subsection{Existence in Region 2}
In this section, based on the a priori estimates in last section, we first show that \eqref{Main Equation} with data prescribed on $C_{u_0}$ where $u_0 \leq \ub \leq \delta$ in last sections can be solved all the way up the $t=-1$, i.e. the Region $2$. Recall that Region $2$ is in the future domain of dependence of $\Cb_0$ and $C_{u_0}$ (with $0 \leq \ub \leq \delta$) and the data on $\Cb_0$ is completely trivial.

To start, we use the local existence result \cite{R-90} of Rendall for semi-linear wave equations for characteristic data, we know that there exists a solution around $S_{0, u_0}$, say, defined in the region enclosed by $\Cb_{0}$, $C_{u_0}$ and $t = u_0 + \varepsilon$ with $\varepsilon << \delta$. Thanks to the a priori estimates, if at the beginning we assume the bound on data for at least $10$ derivatives, the $L^\infty$ norms of at least up to $8$ derivatives of the solution are bounded by the data on $t = u_0 + \varepsilon$. Therefore, we can solve a Cauchy problem with data prescribed on $t = u_0 + \varepsilon$ to construct a solution in the future domain dependence of $t = u_0 + \varepsilon$ whose boundary consists of two null hypersurfaces $C_{u_0 + \varepsilon}$ and $\Cb_{\varepsilon}$. Now we have two characteristic problem: for the first one, the data is prescribed on $\Cb_0$ and $C_{u_0 + \varepsilon}$; for the second one, the data is prescribed on $C_{u_0}$ and $\Cb_{\varepsilon}$. We can use Rendall's local existence result again to solve them around $S_{0, u_0+\varepsilon}$ and $S_{\varepsilon, u_0}$. In this way, we can actually push the solution to $t = u_0 + \varepsilon + \varepsilon'$ with another small $\varepsilon'$.

We then can repeat the above process in an obvious way to push the solution all the way to $t = u_0 + \delta$. Similarly, we can then push it from $t = u_0 + \delta$ to $t= -1$. Therefore, we have constructed a solution in the entire Region $2$. We remark that this process depends crucially on the a priori estimates since the $L^\infty$ norms of the derivatives of $\varphi$ is guaranteed to be bounded.

Therefore, for a finite $u_0$, a solution has been constructed in Region 2.\\

If we restrict the above solution to $\Cb_{\delta}$, i.e. the future incoming null boundary of Region $2$, it gives partially the initial data for \eqref{Main Equation} in Region $3$. We now give a detailed description of the data on $\Cb_{\delta}$.
 \begin{proposition}\label{data on Cb_delta}   Assume we have bounds on $E_i(u_0,\delta)$ and $F_i(u_0,\delta)$ for $i \leq n+2$ for some fixed $n \geq 10$ (say n=10). Then, for all $p \geq 1$ and $q$ with $p+q \leq n-1$, we have
\begin{equation*}
\begin{split}
\|\nablaslash \Omega^q \varphi\|_{L^\infty(\Cb_{\ub})}&\lesssim \delta^{\frac{1}{2}}|u|^{-2}, \quad \text{for all} \quad \ub \in [0, \delta];\\
\|\Lb^{p} \Omega^q \varphi\|_{L^\infty(\Cb_{\ub})}&\lesssim \delta^{\frac{1}{2}}|u|^{-p-1}, \quad \text{for all} \quad \ub \in [0, \delta];\\
\|L^p\Omega^{q} \varphi\|_{L^\infty(\Cb_{\delta})}&\lesssim \delta^{\frac{1}{2}}|u|^{-1}.
\end{split}
\end{equation*}
\end{proposition}
\begin{remark}
The smallness on $L^p \Omega^{q} \varphi$ only holds on the final hypersurface $\Cb_{\delta}$.
\end{remark}
\begin{proof}
First of all, by losing one derivatives, we can retrieve a better $L^{2}(S_{\ub, u})$ estimates for $\nablaslash\Omega^k\varphi$ than in the relaxed propagation estimates.
For $k\leq n$, let
\begin{equation*}
h(\ub, u)\triangleq \|\nablaslash\Omega^k\varphi\|_{L^{2}(S_{\ub,u})},
\end{equation*}
thus,
\begin{equation*}
\begin{split}
\Lb h^{2}(\ub, u)&=\int_{S_{\ub,u}}\Lb(\nablaslash\Omega^k\varphi)^{2}
-\frac{2}{r}\int_{S_{\ub,u}}(\nablaslash\Omega^k\varphi)^{2}\\
& \leq 2\|\Lb\nablaslash\Omega^k\varphi\|_{L^{2}(S_{\ub,u})}\cdot h(\ub, u).
\end{split}
\end{equation*}
where the factor $-\dfrac{1}{r}$ appearing in the second integral is the mean curvature of the incoming hypersurface $\Cb_{\delta}$. Therefore, we have
\begin{equation*}
\frac{\partial}{\partial u}h(\ub, u)\leq \|\Lb\nablaslash\Omega^k\varphi\|_{L^{2}(S_{\ub,u})}.
\end{equation*}
We now integrate to derive
\begin{equation*}
\begin{split}
h(\ub, u)&\leq h(\ub, u_{0})+ \int_{u_0}^{u}\frac{1}{|u'|}\||u'|\Lb\nablaslash\Omega^k\varphi\|_{L^{2}(S_{\ub,u'})}du'\\
&\lesssim h(\ub, u_{0})+ (\int_{u_0}^{u}\frac{1}{|u'|^2} du')^{\frac{1}{2}}\|\Lb\Omega^{k+1}\varphi
\|_{L^{2}(\Cb_{\ub})}\\
&\lesssim \delta^{\frac{1}{2}}|u|^{-2}.
\end{split}
\end{equation*}
For the last line, we have used the bound on the data on $C_{u_0}$ instead of the relaxed propagation estimates. It is precisely at this point one retrieves a $\delta^{\frac{1}{2}}$.

As a consequence, we also retrieve an improved $L^{2}(C_{u})$ and $L^{2}(\Cb_{\ub})$ estimates for $\nablaslash\Omega^k\varphi$:
\begin{equation}\label{Improved L2-nablaslash-varphi}
\begin{split}
\|\nablaslash\Omega^k\varphi\|_{L^{2}(C_{u})}&=(\int_{0}^{\delta}\|\nablaslash\Omega^k\varphi\|^2_{L^2(S_{\ub,u})}d\ub)^\frac{1}{2} \lesssim \delta |u|^{-\frac{3}{2}},\\
\|\nablaslash\Omega^k\varphi\|_{L^{2}(\Cb_{\ub})}&=(\int_{u_{0}}^{u}\|\nablaslash\Omega^k\varphi\|^2_{L^2(S_{\ub,u})}du' )^\frac{1}{2}\lesssim \delta^{\frac{1}{2}}|u|^{-1}.
\end{split}
\end{equation}
According to the Sobolev inequalities, we then obtain
\begin{equation*}
    \|\nablaslash\Omega^q \varphi\|_{L^\infty(\Cb_{\ub})}\lesssim \delta^{\frac{1}{2}}|u|^{-\frac{3}{2}}.
\end{equation*}
This proves the first inequality.\\

For the second inequality, we simply integrate $L(\Lb^{p+1} \Omega^q \varphi)$. To illustrate the idea, we only consider the case where $q=0$. The other cases can be treated exactly in the same way. We commute $\Lb$ with \eqref{Main Equation} $p$ times to derive
\begin{equation}\label{commute Lb n times with equation}
\begin{split}
 L|\Lb^{p+1}  \varphi|& \leq |L\Lb^{p+1}  \varphi|\\
 &\lesssim (\frac{1}{r}+|L\varphi|+|\nablaslash \varphi|)|\Lb^{p+1}  \varphi|+(\frac{1}{r}+|\Lb\varphi|+|\nablaslash \varphi|)|L\Lb^{p}\varphi|+|\nablaslash \Lb^p\varphi||\nablaslash \varphi|  + \text{l.o.t.}.
\end{split}
\end{equation}
where the lower order terms \text{l.o.t.} can be determined inductively. For example, when $p=0,$
\begin{equation*}
\text{l.o.t.} = |\laplacianslash \varphi|.
\end{equation*}
when $p=1,$
\begin{equation*}
\begin{split}
\text{l.o.t.} &\lesssim |Q(\nabla \Lb\varphi, \nabla \varphi)| +\frac{1}{r} |Q(\nabla \varphi, \nabla \varphi)| \\
&+\frac{1}{r^2}(|L\varphi|+|\Lb\varphi|)+|\laplacianslash \Lb\varphi|+\frac{1}{r}|\laplacianslash \varphi|.
\end{split}
\end{equation*}
when $p=2,$
\begin{equation*}
\begin{split}
\text{l.o.t.} &\lesssim |Q(\nabla \Lb^2\varphi, \nabla \varphi)|+|Q(\nabla \Lb\varphi, \nabla \Lb\varphi)|+\frac{1}{r} |Q(\nabla \Lb\varphi, \nabla \varphi)|  + \frac{1}{r^2} |Q(\nabla \varphi, \nabla \varphi)| \\
&+\frac{1}{r^2}(|L\Lb\varphi|+|\Lb^2\varphi|)+\frac{1}{r^3}(|L\varphi|+|\Lb\varphi|)+|\laplacianslash \Lb^2\varphi|+\frac{1}{r}|\laplacianslash \Lb\varphi|+\frac{1}{r^2}|\laplacianslash \varphi|.
\end{split}
\end{equation*}
For $p = 0$, we have the following estimates
\begin{equation*}
\begin{split}
 \|\text{l.o.t.}\|_{L^\infty(\Cb_{\ub})}&\lesssim \delta^{\frac{1}{2}}|u|^{-2},\\
    \|(\frac{1}{r}+|\Lb\varphi|+|\nablaslash \varphi|) L\varphi\|_{L^\infty(\Cb_{\ub})}&\lesssim\delta^{-\frac{1}{2}}|u|^{-2},\\
    |\nablaslash\varphi|^2&\lesssim \delta^{\frac{1}{2}}|u|^{-2}.
\end{split}
\end{equation*}
We integrate along $L$ and use Gronwall's inequality to yield the correct estimates for $\Lb \varphi$. That is
\begin{equation*}
|\Lb \varphi|\lesssim \delta^{\frac{1}{2}}|u|^{-2}.
\end{equation*}
 Moreover, substitute this result into \eqref{commute Lb n times with equation}, it leads to
\begin{equation*}
|L\Lb \varphi|\lesssim \delta^{-\frac{1}{2}}|u|^{-2}.
\end{equation*}
In general, we proceed inductively, assume
\begin{equation*}
\begin{split}
  |\Lb^p \varphi|&\lesssim \delta^{\frac{1}{2}}|u|^{-p-1},\\
  |L\Lb^p \varphi|&\lesssim \delta^{-\frac{1}{2}}|u|^{-p-1}.
  \end{split}
\end{equation*}
Then for $p+1$ we have
\begin{equation*}
\begin{split}
 \|\text{l.o.t.}\|_{L^\infty(\Cb_{\ub})}&\lesssim \delta^{-\frac{1}{2}}|u|^{-p-2},\\
    \|(\frac{1}{r}+|\Lb\varphi|+|\nablaslash \varphi|)L\Lb^{p}\varphi\|_{L^\infty(\Cb_{\ub})}&\lesssim \delta^{-\frac{1}{2}}|u|^{-p-2}\\
    \||\nablaslash \Lb^p\varphi||\nablaslash \varphi|\|_{L^\infty(\Cb_{\ub})}&\lesssim \delta^{\frac{1}{2}}|u|^{-p-2}.
\end{split}
\end{equation*}
We then integrate along $L$ and use Gronwall's inequality to conclude. This proves the second inequality.\\

The third inequality is a little bit surprising, since on expects $L$ derivative causes a loss of $\delta^{-\frac{1}{2}}$, which can been directly from the special choice of the profile of the initial data. The idea is that the loss in $\delta$ only should occur from initial data but not from the energy estimates. Recall that the data is given by
\begin{equation*}
 \varphi(\ub, u_0, \theta) = \frac{\delta^{\frac{1}{2}}}{|u_0|}\psi_0 (\frac{\ub}{\delta}, \theta) ,
\end{equation*}
where the $\ub$-support of $\psi_0$ is inside $(0,1)$. Therefore, on $C_{u_0}$ near $S_{0,u_0}$, the data is completely trivial. In particular, $(L^i \Omega^j \varphi)(u_0, \delta, \theta) \equiv 0$. We then integrate \eqref{Main Equation} to get estimates on $\Cb_{\delta}$.

To illustrate the above idea, we now prove
\begin{equation*}
\|L \varphi\|_{L^\infty(\Cb_{\delta})} \lesssim \delta^{\frac{1}{2}}|u|^{-1}.
\end{equation*}

We rewrite \eqref{Main Equation} as
\begin{equation*}
 \Lb L \varphi  \leq  \frac{1}{r}|L\varphi| + |\Lb\varphi||L \varphi|  + |\nablaslash \varphi| |L\varphi| + \text{l.o.t.}
\end{equation*}
This can be view as an ODE for $L\varphi$ on $\Cb_\delta$ with trivial data on $S_{\delta,u_0}$. We observe that $\nablaslash^i \varphi$ and $\Lb \varphi$ are all of size $\delta^{\frac{1}{2}}|u|^{-2}$ and the l.o.t. is of size $\delta^{\frac{1}{2}}|u|^{-2}$  we can integrate above equation and use Gronwall's inequality to derive
\begin{equation*}
\|L \varphi\|_{L^\infty(\Cb_{\delta})} \lesssim \delta^{\frac{1}{2}}|u|^{-1}.
\end{equation*}
and
\begin{equation*}
\|L\Lb \varphi\|_{L^\infty(\Cb_{\delta})} \lesssim \delta^{\frac{1}{2}}|u|^{-2}.
\end{equation*}
In the same way, by induction, we deduce
\begin{equation*}
\begin{split}
    \|L^p\varphi\|_{L^\infty(\Cb_{\delta})}&\lesssim \delta^{\frac{1}{2}}|u|^{-1}, \\
\|\Lb L^p\varphi\|_{L^\infty(\Cb_{\delta})}&\lesssim \delta^{-\frac{1}{2}}|u|^{-2}.
\end{split}
\end{equation*}
When $q \geq 1$, we can proceed in the same manner. This completes the proof.
\end{proof}

\subsection{Existence in Region 3}
To show the existence of solution for \eqref{Main Equation}, we have to solve a small data problem with data prescribed on $\Cb_{\delta}$ and $C^+_{u_0}$. The data on $\Cb_{\delta}$ is induced from the solution in Region 2 and the smallness of $\delta$ leads to the smallness of the data; the data on $C^+_{u_0}$ is simply an extension by zero of the short pulse data prescribed on $C_{u_0}$, since $L\varphi$ and all higher order derivatives of $\varphi$ on $S_{\delta, u_0}$ are small (we have seen this in the proof of Proposition \ref{data on Cb_delta}), the data in on $C^+_{u_0}$ are also small.

We now prove a theorem similar to the classical small data results \cite{K-84} and \cite{K-85} of Klainerman. The approach we are going to use is inspired by the harmonic gauge based proof of nonlinear stability of Minkowski space-time from \cite{L-R-05} and \cite{L-R-10} of Lindblad and Rodnianski. Since all the arguments are more or less well-known and scattered in the literatures, we only sketch the key estimates.

The following picture will be helpful for the structure of the proof:

\includegraphics[width=4.5  in]{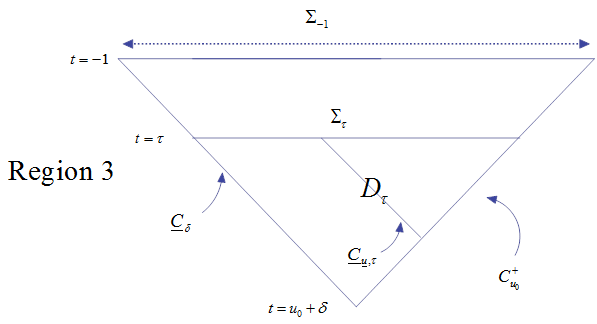}

Let $\Gamma$ denote one of the following vector fields:
\begin{equation*}
\Gamma \in \{\frac{\partial}{\partial t}, \frac{\partial}{\partial x_i}, \Omega_{ij}, \Omega_{0i}, S| i=1,2,3 \},
\end{equation*}
where
\begin{equation*}
\begin{split}
\Omega_{ij} =-\Omega_{ji}& = x_i \frac{\partial}{\partial x_j}-x_j\frac{\partial}{\partial x_i},\\
\Omega_{0i} =\Omega_{i0} &= t \frac{\partial}{\partial x_i}+x_i\frac{\partial}{\partial t},\\
S&=t \frac{\partial}{\partial t}+ r \frac{\partial}{\partial r} = \ub L + u\Lb.
\end{split}
\end{equation*}
For $\varphi$, we define the $k$-th conformal energy on $\Sigma_t$ as follows
\begin{equation*}
E_k^{(\text{conf})}(\varphi)(t) = (\sum_{|\alpha| \leq k} \int_{\Sigma_t} |\partial_t \Gamma^\alpha \varphi|^2 + \sum_{i=1}^3|\partial_i \Gamma^\alpha \varphi|^2 )^\frac{1}{2},
\end{equation*}
where $\Sigma_\tau = \{t = \tau \} \cap \,\text{Region}\,\,2$.

We also use $\partial \phi$ to denote all possible derivatives, i.e. $\partial \phi \in \{\nablaslash \phi, L\phi, \Lb \phi\}$; we use $\bar{\partial} \phi$ to denote good derivatives, i.e. $\bar{\partial} \phi \in \{\nablaslash \phi, \Lb \phi\}$.

Therefore, the classical Klainerman-Sobolev inequalities implies
\begin{equation*}
\|\partial \phi \|_{L^\infty(\Sigma_t)} \lesssim \frac{1}{(1+|u|)(1+|\ub|)^\frac{1}{2}} E_3^{(\text{conf})}(\phi).
\end{equation*}

When one commutes $\Gamma$'s with \eqref{Main Equation} $k$ times (k=8 suffices), we have
\begin{equation}\label{main equation gamma k}
\Box \Gamma^k \varphi = F_k
\end{equation}
where
\begin{equation*}
F_k = \sum_{|\beta| + |\gamma| \leq k} Q(\nabla \Gamma^\beta \varphi, \nabla \Gamma^\gamma \varphi),
\end{equation*}
where the $Q$'s are null forms and we have ignored the irrelevant numerical constants.

Once the following a priori estimates have been established, the rest of the proof will be routine. So we only give the details for the following estimates:
\begin{proposition}
For all $k \leq 8$, assume that the standard energy fluxes on $\Cb_{\delta}$ for \eqref{main equation gamma k} are all bounded above by $\varepsilon$ where $\varepsilon$ is a small positive constant. We also assume that such a solution of \eqref{Main Equation} exists in $\mathcal{D}_t$ where $u_0 + \delta \leq t \leq -1$ and $\mathcal{D}_t$ is the region below $\Sigma_t$. If $\varepsilon$ is sufficiently small, there is a universal constant $C_0$, such that
\begin{equation*}
\sum_{k=0}^8 \sup_{\tau \in [u_0 + \delta, t]} E_k^{(\text{conf})}(\varphi)(\tau) \leq C_0 \cdot \varepsilon.
\end{equation*}
\end{proposition}
\begin{proof}
First of all, using Killing vector field $\dfrac{\partial}{\partial_t}$ in $\mathcal{D}_t$, we have the following standard energy estimates:
\begin{equation}\label{standard energy estimates in D_t}
\int_{\Sigma_t} |\partial_t \Gamma^k\varphi|^2 + \sum_{i=1}^3|\partial_i\Gamma^k\varphi|^2 \leq \varepsilon^2 + 2\int_{u_0+\delta}^{t}\int_{\Sigma_\tau}|F_k||\partial_t \Gamma^k\varphi|.
\end{equation}
Secondly, let $\Cb_{\ub, t}$ is the part of $\Cb_{\ub}$ in $\mathcal{D}_t$. We can then apply the standard energy estimates in the region bounded by $\Cb_{\ub, t}$, $\Cb_{\delta}$, $\Sigma_t$ and $C^+_{u_0}$, we have
\begin{equation*}
\int_{\Cb_{\ub, t}} |\bar{\partial}\Gamma^k\varphi|^2 \leq \varepsilon^2 + 2\int_{u_0+\delta}^{t}\int_{\Sigma_\tau}|F_k||\partial_t \Gamma^k\varphi|.
\end{equation*}
We multiply the above inequality by $\dfrac{1}{(1+|\ub|)^{1+\kappa}}$ with $\kappa>0$ and integrate over $\ub \in [\delta, t-u_0]$, this yields
\begin{align*}
\doubleint_{\mathcal{D}_t} \frac{|\bar{\partial}\Gamma^k\varphi|^2}{(1+|\ub|)^{1+\kappa}} &=  \int_{\delta}^{t-u_0} \int_{\Cb_{\ub, t}} \frac{|\bar{\partial}\Gamma^k\varphi|^2}{(1+|\ub|)^{1+\kappa}}\\
&\lesssim \varepsilon^2 + \int_{u_0+\delta}^{t}\int_{\Sigma_\tau}|F_k||\partial_t \Gamma^k\varphi|.
\end{align*}
Combined with \eqref{standard energy estimates in D_t}, we obtain the following estimates which will serve as the main tool for the rest of the proof:
\begin{equation}\label{main energy estimates in small data case}
\sum_{|\alpha| \leq 8}\int_{\Sigma_t} |\partial \Gamma^\alpha \varphi|^2 +  \sum_{|\alpha| \leq 8}\doubleint_{\mathcal{D}_t} \frac{|\bar{\partial}\Gamma^\alpha\varphi|^2}{(1+|\ub|)^{1+\kappa}} \lesssim \varepsilon^2 + \sum_{|\alpha| \leq 8, k\leq 8}\int_{u_0+\delta}^{t}\int_{\Sigma_\tau}|F_k||\partial_t \Gamma^\alpha\varphi|.
\end{equation}
We now use a bootstrap argument and we make the following bootstrap assumptions:
\begin{equation*}
\sum_{k=0}^8 \sup_{\tau \in [u_0 + \delta, t]} E_k^{(\text{conf})}(\varphi)(\tau) \leq M \cdot \varepsilon.
\end{equation*}
where $M$ is a large constant.\\

Since $F_k$'s are linear combinations of null forms, the nonlinear terms on the right hand side of \eqref{main energy estimates in small data case} can be bounded by $N_1 + N_2$ where
\begin{equation*}
\begin{split}
N_1 &= \sum_{|\beta|+|\gamma| \leq 8, k\leq 8, |\gamma|\leq 4}\int_{u_0+\delta}^{t}\int_{\Sigma_\tau} |\partial \Gamma^\beta\varphi||\bar{\partial} \Gamma^\gamma\varphi||\partial_t \Gamma^\alpha\varphi|,\\
N_2 &= \sum_{|\beta|+|\gamma| \leq 8, k\leq 8, |\beta|\leq 4}\int_{u_0+\delta}^{t}\int_{\Sigma_\tau} |\partial \Gamma^\beta\varphi||\bar{\partial} \Gamma^\gamma\varphi||\partial_t \Gamma^\alpha\varphi|.
\end{split}
\end{equation*}

The control of $N_1$ is slightly easier, we use Klainerman-Sobolev inequalities for $|\bar{\partial} \Gamma^\gamma\varphi|$. Most importantly, recall that the Klainerman-Sobolev inequalities improve a factor of $\dfrac{1}{|u|^{\frac{1}{2}}}$ for good derivatives. This yields
\begin{align*}
N_1 &\lesssim \sum_{|\beta|+|\gamma| \leq 8, k\leq 8, |\gamma|\leq 4}\int_{u_0+\delta}^{t}\frac{M\cdot \varepsilon}{|u|^\frac{3}{2}}(\int_{\Sigma_\tau} |\partial \Gamma^\beta\varphi||\partial_t \Gamma^\alpha\varphi|)d\tau\\
&\lesssim \sum_{|\beta|+|\gamma| \leq 8, k\leq 8, |\gamma|\leq 4}\int_{u_0+\delta}^{t}\frac{M\cdot \varepsilon}{|u|^\frac{3}{2}}(M\cdot \varepsilon)^2 d\tau\\
&\lesssim M^3 \cdot \varepsilon^3.
\end{align*}

To control $N_2$, we would like to use the second flux term in \eqref{main energy estimates in small data case}. To this end, we first use Cauchy-Schwarz inequality:
\begin{align*}
N_2 & \lesssim\frac{1}{C} \sum_{|\beta|+|\gamma| \leq 8, k\leq 8, |\beta|\leq 4}\int_{u_0+\delta}^{t}\frac{|\bar{\partial}\Gamma^\alpha\varphi|^2}{(1+|\ub|)^{1+\kappa}}\\
 \quad &+C \sum_{|\beta|+|\gamma| \leq 8, k\leq 8, |\beta|\leq 4}\int_{u_0+\delta}^{t}\int_{\Sigma_\tau} (1+|\ub|)^{1+\kappa} |\partial \Gamma^\beta\varphi|^2|\partial_t \Gamma^\alpha\varphi|^2
\end{align*}
We can choose a large $C$ so that the first term can be absorbed by the left hand side of \eqref{main energy estimates in small data case}; for the second term, we can use  Klainerman-Sobolev inequalities to control $|\partial \Gamma^\beta\varphi|^2$ and the energy norms to control the rest, this yields
\begin{align*}
N_2 &\lesssim M^4 \cdot \varepsilon^4.
\end{align*}

We put all these estimates back to \eqref{main energy estimates in small data case}, this yields
\begin{equation*}
\sum_{|\alpha| \leq 8}\int_{\Sigma_t} |\partial \Gamma^\alpha \varphi|^2 \lesssim \varepsilon^2 +M^3 \cdot \varepsilon^3+ M^4 \cdot \varepsilon^4.
\end{equation*}

We can then take a sufficiently small $\varepsilon$ to close the argument.
\end{proof}

The parameter $\varepsilon$ is proportional to $\delta^\frac{1}{2}$. Let $\delta$ be sufficiently small, therefore, we have constructed a solution in Region 3.

\subsection{Existence from Past Null Infinity}

We will let $u_0$ go to $-\infty$ so that the null hypersurface $C_{u_0}$ will approximate the past null infinity.\\

We choose a decreasing sequence $\{u_{0, i}\}$ in such a way that $u_{0, i} \rightarrow -\infty$ and we solve Goursat problem for \eqref{Main Equation} with initial data at $C_{u_{0, i}}$. We emphasize that the choice of the seed data $\psi_0$ is the same for all $u_{0, i}$'s. For each $u_{0, i}$, we obtain a unique smooth solution $\varphi_i$ defined in the region $\mathcal{D}_i$ where
\begin{equation*}
\mathcal{D}_i = \{p \in \mathbb{R}^{3+1}| t(p) \leq -1, u(p) \geq u_{0,i}\}.
\end{equation*}
Moreover, by Sobolev inequalities, for all $k \leq 8$, there exists a constant $C_0$ independent of $i$, so that
\begin{equation*}
\|\varphi_i\|_{C^k(\mathcal{D}_i)} \leq C_0.
\end{equation*}
Thanks to the lemma of Arzela-Ascoli, we could extract a subsequence of solutions, still denoted by $\{\varphi_i\}$, converges uniformly in any compact set of $\{(t,x)\in \mathbb{R}^{3+1}| t \leq -1\}$. We denote the limit to be $\varphi$ and this is a classical solution of \eqref{Main Equation} all the way up to past null infinity.\\

To prove the \textbf{Main Theorem 2} of the paper (it implies the \textbf{Main Theorem 1} in a straightforward way), it remains to show the uniqueness from past null infinity. Suppose $\varphi$ and $\phi$ were two classical solutions for \eqref{Main Equation} with the same scattering data \eqref{Initial-data-from-past-null}, therefore,
\begin{equation}\label{w-v}
\Box (\varphi-\phi)=Q(\nabla \varphi, \nabla \varphi)-Q(\nabla \phi, \nabla \phi)\triangleq F(\nabla\varphi, \nabla\phi),
\end{equation}
with
\begin{equation}\label{initial-data-w-v}
\lim_{u_0\rightarrow-\infty}|u_0(\varphi-\phi)|_{u=u_0}=0.
\end{equation}
For $\tau \leq -1$, let
\begin{equation*}
E(\tau)=\int_{t= \tau}|\partial_t(\varphi-\phi)|^2+\Sigma_{i=1}^3 |\partial_i (\varphi-\phi)|^2,
\end{equation*}
the condition \eqref{initial-data-w-v} reads  as
\begin{equation*}
\lim_{t \rightarrow -\infty}E(t)=0,
\end{equation*}
i.e. $E(t) = o(1)$. According to the standard energy estimates, we have
\begin{equation}\label{energy-method-sigma-t}
E(t)\leq E(2t) + \int_{2t}^{t} (\int_{\Sigma_\tau} |F(\nabla\varphi, \nabla\phi)||\partial_t(\varphi-\phi)|)d\tau.
\end{equation}
We observe that at least factor in the quadratic form $F(\varphi, \phi)$ is a good term, i.e. it is either $\bar{\partial}\varphi$, $\bar{\partial}\phi$ or $\bar{\partial}(\varphi-\phi)$. We use $L^\infty$ bound for this term and it is bounded by $O(|t|^{-\frac{3}{2}})$. Therefore, \eqref{energy-method-sigma-t} becomes
\begin{equation}\label{interation}
\begin{split}
E(t)&\leq E(2t) + C_0 \int_{2t}^{t}|\tau|^{-\frac{3}{2}}E(\tau)d\tau\\
& \leq  E(2t) + C_0 |t|^{-\frac{1}{2}} \sup_{\tau \in [2t,t]}E(\tau)\\
& \leq  E(2t) + C_0^2|t|^{-\frac{1}{2}},
\end{split}
\end{equation}
where $C_0$ is the size of $E(t),$ which is actually small.
We can iterate this estimate to derive (notice that the following $o(1)$ is independent of $k$)
\begin{equation*}
E(t)\leq E(2^k \cdot t) +(\sum_{j=0}^{k-1}2^{-\frac{j}{2}}) C_0^2 |t|^{-\frac{1}{2}}, \quad \forall \,\,\,k \in \mathbb{Z}_{>0}.
\end{equation*}
We then let $k \rightarrow \infty$ and we improve the decay $E(t)= o(1)$ to be
\begin{equation}\label{Decay-E-t-1}
E(t)\leq \frac{C_0^2}{1-2^{-\frac{1}{2}}}|t|^{-\frac{1}{2}}.
\end{equation}

We then substitute \eqref{Decay-E-t-1} into \eqref{interation} to iterate again, we obtain
\begin{equation*}
E(t)\leq E(2t) + \frac{C_0^3}{2(1-2^{-\frac{1}{2}})}|t|^{-1}.
 \end{equation*}
We repeat the above dyadic iteration and we can further improve the decay of $E(t)$ to be
\begin{equation*}
E(t)\leq  \frac{C_0^3}{1-2^{-\frac{1}{2}}}|t|^{-1}.
\end{equation*}

We the repeat the whole procedure several times and we can obtain, for all $k \in \mathbb{Z}_{>0}$,
\begin{equation*}
E(t)\leq  \frac{C_0^{3+k}}{(1-2^{-\frac{1}{2}})\prod_{j=0}^k|1+\frac{j}{2}|}|t|^{-1-\frac{k}{2}}.
\end{equation*}\
Let $k \rightarrow \infty$, it implies $E(t)=0$. Therefore, we have proved the uniqueness.

\subsection{Proof of Main Theorem 3}
It remains to prove the \textbf{Main Theorem 3} and this is an easy consequence of the estimates derived previously. We first recall that $u_0$ is a fixed number. From the proof of Proposition \ref{data on Cb_delta}, we know that, for all $\theta \in \mathbb{S}^2$, $\ub \in [0,\delta]$ and $u\in[u_0,-1]$, we have
\begin{equation*}
\begin{split}
|L\varphi(u,\ub, \theta)- L \varphi(u_0,\ub, \theta)| &\lesssim \delta^{\frac{1}{2}},\\
|\nablaslash\varphi(u,\ub, \theta)- \nablaslash\varphi(u_0,\ub, \theta)| &\lesssim \delta^{\frac{1}{2}},\\
|\Lb\varphi(u,\ub, \theta)- \Lb\varphi(u_0,\ub, \theta)| &\lesssim \delta^{\frac{1}{2}}.
\end{split}
\end{equation*}
Since we take the short pulse data localized in $B_{\delta^{\frac{1}{2}}}(\theta_0)$,  we simply integrate the above inequalities and use the pointwise information of the initial data on $C_{u_0}$, this shows immediately
\begin{equation*}
\begin{split}
 \int_{C_u -  C^o_{u}}|L \varphi|^2 + |\nablaslash \varphi|^2 &\lesssim \delta^{2},\\
 |\int_{C^o_{u}} |L \varphi|^2 + |\nablaslash \varphi|^2 -  \int_{C^o_{u_0}} |L \varphi|^2 + |\nablaslash \varphi|^2 |&\lesssim \delta.
\end{split}
\end{equation*}
Moreover, on $\Cb_{\delta}$, we have seen that almost no energy radiating from it, in terms of estimates, this means $\Eb_1(u,\delta) \lesssim \delta^{\frac{1}{2}}$. This demonstrates the strong focused phenomenon and completes the proof of the \textbf{Main Theorem 3}.


\begin{thebibliography}{99}

\bibitem{Ch-86} D. Christodoulou, \textit{Global solutions of nonlinear hyperbolic equations for small initial data}, Comm. Pure Appl. Math., 39(2), 267-282.

\bibitem{Ch-08} D. Christodoulou, \textit{The Formation of Black Holes in General Relativity}, Monographs in Mathematics, European Mathematical Soc. 2009.

\bibitem{Ch-K} D. Christodoulou and S. Klainerman, \textit{The Global Nonlinear Stability of Minkowski space}, Princeton Mathematical Series 41, 1993.


\bibitem{J-79} F. John, \textit{Blow-up of solutions of nonlinear wave equations in three space dimensions}. Manuscripta Math., 28(1-3), 235-268, 1979.

\bibitem{K-80} S. Klainerman, \textit{Global existence for nonlinear wave equations}, Comm. Pure Appl. Math., 33(1):43-101, 1980.

\bibitem{K-84} S. Klainerman, \textit{Long time behaviour of solutions to nonlinear wave equations},  Proceedings of the International Congress of Mathematicians, 1209 - 1215, Warsaw, 1984. PWN.

\bibitem{K-85} S. Klainerman. \textit{Uniform decay estimates and the Lorentz invariance of the classical wave equation}, Comm. Pure Appl. Math., 38(3):321-332, 1985.

\bibitem{K-R-09} S. Klainerman and I. Rodnianski, \textit{On the Formation of Trapped Surfaces}, Acta Math. 208 (2012), no. 2, 211每333.

\bibitem{K-R-10} S. Klainerman and I. Rodnianski, \textit{On emerging scarred surfaces for the Einstein vacuum equations}, Discrete Contin. Dyn. Syst. 28 (2010), no. 3, 1007每1031.

\bibitem{L-R-05} H. Lindblad and I. Rodnianski, \textit{Global existence for the Einstein vacuum equations in wave coordinates}, Comm. Math. Phys. 256 (2005), no. 1, 43每110.

\bibitem{L-R-10} H. Lindblad and I. Rodnianski, \textit{The global stability of Minkowski space-time in harmonic gauge}, Ann. of Math. (2) 171 (2010), no. 3, 1401每1477.

\bibitem{R-90} A. D. Rendall, \textit{Reduction of the characteristic initial value problem to the Cauchy problem and its applications to the Einstein equations}, Proc. Roy. Soc. Lond. A 427, 221 - 239 (1990).

\end{thebibliography}
\end{document}